\documentclass[a4paper, 11pt]{article}
\usepackage{latexsym,amsfonts, amsmath, amssymb, amsthm}
\usepackage[utf8]{inputenc}
\usepackage[final=true,protrusion=true,expansion=true]{microtype}
\usepackage[noadjust]{cite}
\usepackage{enumitem}
\usepackage{color}
\usepackage{mathtools}

\usepackage{booktabs}
\usepackage{multirow}

\usepackage{subcaption}
\usepackage[font=scriptsize,labelfont=bf]{caption}
\usepackage[affil-it]{authblk}

\usepackage{geometry}
\usepackage{hyperref}
\hypersetup{
	final,
    colorlinks=true,
    linkcolor=blue,
    filecolor=magenta,
    urlcolor=cyan,
}

\newcounter{results}[section]

\theoremstyle{plain}
\newtheorem{theorem}[results]{Theorem}

\newtheorem{lemma}[results]{Lemma}

\newtheorem{assumption}[results]{Assumption}
\newtheorem{remark}[results]{Remark}

\theoremstyle{definition}
\newtheorem{definition}[results]{Definition}

\usepackage{amssymb,amsmath,amscd,amsfonts,amsthm,bbm,mathrsfs,yhmath}
\usepackage{hyperref}
\usepackage{url}
\usepackage[utf8]{inputenc} 
\usepackage{cleveref}
\usepackage{algorithm}
\usepackage{algpseudocode}
\usepackage{graphicx}
\usepackage{color}
\usepackage{listings}
\usepackage{physics}
\usepackage{nicefrac}
\usepackage{subcaption}
\usepackage{geometry}
\usepackage{multirow}
\usepackage[table]{xcolor}
\usepackage{fullpage}
\usepackage[table]{xcolor}

\usepackage[colorinlistoftodos,bordercolor=orange,backgroundcolor=orange!20,linecolor=orange,textsize=scriptsize]{todonotes}

\newif\ifrevised

\newcommand{\rrevised}[1]{%
	\ifrevised
	\color{blue} #1 \color{black} %
	\else
	#1%
	\fi}
\revisedfalse

\renewcommand{\div}{\mathrm{div}}
\newcommand{\intd}{\int_{\mathbb{R}^d}}
\newcommand{\intt}{\int_0^T}
\newcommand{\inner}[2]{\Big\langle#1,#2\Big\rangle}

\newcommand{\intr}{\int_{B_R(0)}}
\renewcommand{\H}{H}
\newcommand{\e}{e}
\newcommand{\Q}{Q}
\newcommand{\J}{\boldsymbol{J}}

\title{
	Regularity and positivity of solutions of the Consensus-Based Optimization equation: unconditional global convergence
}

\author[1,2,3]{Massimo Fornasier\thanks{Email: \texttt{massimo.fornasier@cit.tum.de} }}
\author[1,2]{Lukang Sun\thanks{Email: \texttt{lukang.sun@tum.de} }}

\affil[1]{Technical University of Munich, School of Computation, Information and Technology, Department of Mathematics, Munich, Germany}
\affil[2]{Munich Center for Machine Learning, Munich, Germany}
\affil[3]{Munich Data Science Institute, Germany}

\begin{document}
	\maketitle
	\begin{abstract} 
Introduced in  {{\it R. Pinnau, et al. (Math. Models Methods Appl. Sci., 2017)}}, Consensus-Based Optimization (CBO) has rapidly emerged as a significant breakthrough in global optimization. This straightforward yet powerful multi-particle, zero-order optimization method draws inspiration from Simulated Annealing and Particle Swarm Optimization. Using a quantitative mean-field approximation, CBO dynamics can be described by a nonlinear Fokker-Planck equation with degenerate diffusion, which does not follow a gradient flow structure.  

In this paper, we demonstrate that solutions to the CBO equation remain positive and maintain full support. Building on this foundation, we establish the {\it unconditional} global convergence of CBO methods to global minimizers. Our results are derived through an analysis of solution regularity and the proof of existence for smooth, classical solutions to a broader class of drift-diffusion equations, despite the challenges posed by degenerate diffusion.
	\end{abstract}

	
	\tableofcontents

	\section{Introduction}

Traditional mathematical optimization techniques, such as gradient descent and quasi-Newton methods, rely on local higher-order information to guide the search for {global} minimizers. However, due to their reliance on local information, these approaches often fail when applied to nonconvex and nonsmooth problems, which pose the greatest challenge in mathematical optimization.
Metaheuristic algorithms, including Simulated Annealing, Genetic Algorithms, Particle Swarm Optimization, Ant Colony Optimization, and various evolutionary strategies \cite{GendPotv13,blum2003metaheuristics}, have demonstrated remarkable empirical success and have been extensively evaluated on benchmarking platforms like COCO \cite{COCO21}. Despite their success, most of these methods lack strong theoretical guarantees or reliable convergence rates, particularly for high-dimensional global nonconvex optimization problems. Effectively overcoming nonconvexity and locality barriers with theoretical guarantees would enable us to tackle applications that have, until now, been beyond the reach of {\it rigorous} mathematical approaches.

A major breakthrough in global optimization occurred in 2017 with the introduction of Consensus-Based Optimization (CBO) \cite{B1-pinnau2017consensus}. CBO is a simple, zero-order, multi-particle optimization method inspired by the principles of Simulated Annealing and Particle Swarm Optimization. 
{As with other metaheuristic techniques, the flexibility and simplicity of CBO have played a key role in its increasing use for solving a broad spectrum of optimization problems in fields such as machine learning, engineering, and the sciences. These applications span global optimization with constraints \cite{fornasier2020consensus_hypersurface_wellposedness,B1-fornasier2020consensus_sphere_convergence,bae2022constrained,beddrich2024,B1-carrillo2021consensus,B1-borghi2021constrained,B1-ha2021emergent}, minimizing cost functions with multiple optima \cite{bungert2022polarized,B1-fornasierSun25}, addressing multi-objective formulations \cite{borghi2022consensus,borghi2022adaptive,totzek-multio}, handling stochastic optimization scenarios \cite{bellavia25,bonandin2024}, solving high-dimensional problems in machine learning \cite{B1-carrillo2019consensus}, and sampling from complex distributions \cite{B1-carrillo2022sampling}. The framework has been further refined through the integration of enhancements such as momentum terms \cite{CiCP-31-4}, memory components \cite{B1-totzeck2020consensus,HUANG2025109372,riedl2022leveraging}, second-order dynamics \cite{2ndorder,B1-cipriani2021zero,B1-grassi2021mean,grassi2020particle,hui23}, optimal control mechanisms \cite{huang2024fast}, mirror descent methods \cite{bungert2025mirror}, random batch strategies \cite{ko2022convergence}, truncated noise models \cite{fornasier2023consensus1}, and jump processes \cite{kalise2022consensus}. Moreover, it has been effectively applied to multi-agent games \cite{chenchene2024}, min-max optimization \cite{B1-qiu2022Saddlepoints}, hierarchical or multi-level problems \cite{multilevel,trillos2025cb2o}, and clustered federated learning \cite{JMLR:v25:23-0764}\footnote{The body of literature on CBO and its testing on applications is growing rapidly, making it challenging to track and report all the latest developments. Here, we cite a selection of recent noteworthy contributions. We refer to \cite{totzeck2021trends} for a  review and \cite{totzeck2018numerical} for a comparison with other multi-agent global optimization algorithms.}.
}
CBO not only demonstrates strong empirical success in solving a wide range of global optimization problems  but, even more surprisingly, also comes with rigorous theoretical guarantees and convergence rates for a broad class of nonconvex and nonsmooth optimization problems; see our discussion below for details.

The equations defining the iterates $V^i_k$ of the CBO algorithm read ($i=1,\dots,N$ labelling the particles, and $k=0,\dots,K$ denoting the iterates)
\begin{align}
\label{eq:B1-CBO1}
V^i_{k+1} &= V^i_{k} - \Delta t \lambda(V^i_k-v_{\alpha}(\widehat \rho_k^N)) + \sqrt{\Delta t}\sigma \|V^i_k-v_{\alpha}(\widehat  \rho_k^N)\|_2  B_k^i, 
\end{align}
where
\begin{align}
v_{\alpha}(\widehat \rho_k^N) = \frac{1 }{\sum_{i=1}^N \omega_\alpha^f(V_k^i)} \sum_{i=1}^N V_k^i \omega_\alpha^f(V_k^i), 
\end{align}
and $f$ is the objective function to be minimized, $v_{\alpha}(\widehat \rho_k^N)$ is the so-called consensus point, $\omega_\alpha^f(v)=e^{-\alpha {f}(v)}$ is  the Gibbs weight, and $B^i_k$ is an independent standard Gaussian random vector. The initial {particles} $V_0^i$ are drawn i.i.d. at random according to a given probability distribution $\rho_0$.

The algorithm \eqref{eq:B1-CBO1} has a straightforward formulation (it can be implemented in just a few lines of code \cite{CBX2024code}) with a complexity of $\mathcal O(N)$ and is highly parallelizable; see, for instance, \cite{benfenati22,JMLR:v25:23-0764}. Additionally, it relies exclusively on pointwise evaluations of the objective function $f$, meaning that no higher-order derivatives are required.  

The core mechanism of CBO involves a combination of a relaxed drift toward the weighted mean of the particles and their stochastic perturbation, embodying an exploitation-exploration strategy. Notably, the original formulation of CBO, given in equation \eqref{eq:B1-CBO1}, was developed under the restrictive assumption that the objective function $f$ possesses a {\it unique} {global} minimizer $v^*$. More recent adaptations extend the approach to accommodate multiple global minimizers \cite{bungert2022polarized,B1-fornasierSun25}.  

The proofs of convergence to global minimizers as devised in \cite{carrillo2018analytical,B1-fornasier2021global} rely on the following three-level approximation argument. The first approximation goes for $\Delta t \to 0$: standard results, see \cite{platen1999introduction}, allow to interpret \eqref{eq:B1-CBO1} as  the Euler--Maruyama numerical approximation  for $\Delta t \to 0$ to solutions $V_t^i$, $i=1,\dots,N$ of the empirical system
\begin{equation}
d {V}^i_t
		=-\lambda\left(V^i_t-v_\alpha(\hat \rho_t^N )\right)dt+\sigma \|V^i_t-v_\alpha(\widehat \rho^N_t)\|_2 d B_t^i, \quad v_\alpha(\widehat \rho_t^N ) =  \frac{1 }{\sum_{i=1}^N \omega_\alpha^f(V_t^i)} \sum_{i=1}^N V_t^i \omega_\alpha^f(V_t^i).
\end{equation}
The second approximation level comes as a quantitative mean-field limit for $N \to \infty$, see \cite[Proposition 3.11]{B1-fornasier2021global} and also \cite[Theorem 2.6]{gerber2023propagation},
\begin{equation}\label{eq:qmfl}
   \sup_{t \in [0,T]}\sup_{i=1,\dots,N} \mathbb E\left [ \| \overline V^i_t -V^i_t\|_2^2\right ]  \leq CN^{-1}, 
\end{equation}
{where $\overline V^i_t$ are $N$ particles satisfying the self-fulfilling mono-particle dynamics
\begin{equation}\label{eq:B1-CBO2}
		d \overline{V}_t
		=-\lambda\left(\overline{V}_t-v_\alpha(\rho_t)\right)dt+\sigma \|\overline{V}_t-v_\alpha(\rho_t)\|_2 d B_t,
	\end{equation}
 with $\overline V^i_0=V_0^i, i=1,\ldots,N$,}   
 and  $v_\alpha(\rho)=\frac{\int v \omega^{f}_\alpha(v) d\rho(v)}{\int \omega^f_\alpha(v) d\rho(v)}$ and $\rho_t= \operatorname{Law}(\overline{V}_t)$, which fulfills the nonlinear Fokker--Planck equation
 \begin{align} \label{eq:B1-fokker_planck}
	\partial_t\rho_t
	= \lambda\operatorname{div} \big(\!\left(v - v_\alpha(\rho_t)\right)\rho_t\big)
	+ \frac{\sigma^2}{2}\Delta\big(\|v-v_\alpha(\rho_t)\|_2^2\rho_t\big),
\end{align}
{here $v\in\mathbb{R}^d$ is the variable of the equation.} Notably, the existence of probability measure solutions of \eqref{eq:B1-fokker_planck} was established in \cite{carrillo2018analytical} \rrevised{through} the analysis of \eqref{eq:B1-CBO2}. 
The third and last approximation level is provided by the large time asymptotics of the solution $\rho_t$. Namely, the global convergence is resulting from the quantification of the approximation $v_\alpha(\rho_t) \approx v^*$ at all times, for $\alpha>0$ large enough, obtained  through a quantitative Laplace's principle \cite[Proposition 4.5]{B1-fornasier2021global} and the crucial technical assumption that $\rho_0$ has mass around 
the global minimizer $v^*$, i.e.,
\begin{equation}\label{eq:techass}
\rho_0(B_r(v^*)) >0, \quad r>0.
\end{equation}
See \cite[Proposition 4.6]{B1-fornasier2021global} for details.
Consequently, under mild assumptions on the objective function $f$ --local Lipschitz continuity and a local growth around the global minimizer--, \cite[Theorem 3.7]{B1-fornasier2021global} shows that the Wasserstein distance $W_2^2(\rho_t, \delta_{v^*})$ is a natural Lyapunov functional for \eqref{eq:B1-fokker_planck} at finite time, with exponential decay
\begin{equation}\label{eq:contr}
W_2^2(\rho_t, \delta_{v^*}) \leq W_2^2(\rho_0, \delta_{v^*}) e^{-(1-\theta)(2\lambda - d \sigma^2) t}, 
\end{equation}
for a suitable $\theta \in (0,1)$ for $t \in [0,T^*]$ and 
\begin{equation}\label{eq:finerr}
    W_2^2(\rho_{T^*}, \delta_{v^*}) \leq \epsilon.
\end{equation}
The combination of these three level approximation as in \cite[Theorem 3.8]{B1-fornasier2021global} yields the quantitative convergence in probability of the algorithm \eqref{eq:B1-CBO1}
\begin{equation}\label{B1-ownresult2}
  \mathbb P\Bigg( \|\frac{1}{N} \sum_{i=1}^N V_{K}^i-v^*\| \leq \varepsilon\Bigg) \geq 1 - \left [ \varepsilon^{-1} (C_{\mathrm{NA}}\Delta t+C_{\mathrm{MFA}} N^{-1}+\epsilon) -\delta \right ].
\end{equation}

At this point, we turn our attention to the technical assumption \eqref{eq:techass}, which was introduced in the earlier work \cite{carrillo2018analytical,B1-fornasier2021global}. This condition constrains the initial distribution \(\rho_0\) used for sampling particles, requiring it to include the global minimizer \(v^*\) in its support (for instance $\rho_0$ may be a Gaussian distribution). However, from the results of practical numerical experiments, this assumption did not seem necessary for the method to achieve global convergence. This suggests that CBO actually converges {\it unconditionally}.
The difficulty in providing an obvious statement about condition \eqref{eq:techass} comes from the degeneracy of the diffusion in equation \eqref{eq:B1-fokker_planck}. 
\\

One of the results of this paper is to prove that, as long as $\rho_0$ is a smooth density--not necessarily requiring \rrevised{$v^* \in \operatorname{supp}(\rho_0)$}-- {and the dimensionality of the problem $d>1$},  the density $\rho_t(v)>0$ for all $v \in \mathbb R^d \setminus \{v_\alpha(\rho_t)\}$ and for all $t>0$, see Theorem \ref{thm:positivity}. Hence, this positivity property renders \eqref{eq:techass} completely superfluous and, combined with \cite[Theorem 3.7, Theorem 3.8]{B1-fornasier2021global}, we can actually establish the {\it unconditional} global convergence of CBO to global minimizers, regardless of whether the initial distribution has mass near the global minimizers.\\

The positivity result we obtain leverages  regularity results of solutions of \eqref{eq:B1-fokker_planck}, in particular we show existence and uniqueness of classical solutions. Although, we believe that one could obtain positivity even without such regularity results, the second goal of this paper is in fact to have a finer description of the solution $\rho_t$. Indeed, while existence of probability measure solutions was established in \cite{carrillo2018analytical}, neither well-posedness nor regularity results were known so far.
We will do that by proving intermediate results for more general (linear) drift-diffusion equations with degenerate diffusion of the type \eqref{eq:g1}-\eqref{eqg:2}. Given their relevance in the development of the paper we consider these findings the main results of the paper. The regularity will be consequence of a nontrivial interplay between the degenerate diffusion and  drift terms. As long as the diffusion term ``dominates" the drift in a controlled manner (see Assumption \ref{asp:g1}), then there is no solution degeneracy occurring.
{Let us conclude this introduction by mentioning that only after the preparation of the present work and in the phase of its revision,  we became aware of the recent study \cite{wang2025}, which also
investigates the well-posedness of \eqref{eq:B1-fokker_planck}. In the latter work the analysis adopts a different methodological approach. The authors of \cite{wang2025} establish the existence of regular solutions via a compactness argument applied to a regularized version of the PDE \eqref{eq:B1-fokker_planck}. In contrast, the results in our present paper achieve higher regularity of solutions for a broader class of PDE and positivity of solutions for \eqref{eq:B1-fokker_planck}.}

Before describing the main results in details, let us introduce the notations of the paper.

\paragraph{Notations.} We use $\mathbb{N}_0^d$ to denote a $d$-dimension vector with each element non-negative integer and $[d]:=\{1,\ldots,d\}$.  If $\boldsymbol{\nu}=\left(\nu_1, \ldots, \nu_d\right) \in \mathbb{N}_0^d$ and
$\mathbf{z}=\left(z_1, \ldots, z_d\right) \in \mathbb{R}^d$,
then $|\boldsymbol{\nu}|  =\sum_{i=1}^d \nu_i$, $\boldsymbol{\nu}! =\prod_{i=1}^d\left(\nu_{i}!\right)$, $h^{(\boldsymbol{\nu})}=\frac{\partial^{|\boldsymbol{\nu}|}}{\partial v_1^{\nu_1} \cdots \partial v_d^{\nu_d}} h, \boldsymbol{w}^{(\boldsymbol{\nu})}=\left(w_1^{(\boldsymbol{\nu})},\ldots,w_m^{(\boldsymbol{\nu})}\right)$, here $\boldsymbol{w}=(w_1,\ldots,w_m)$, $\mathbf{z}^{\boldsymbol{\nu}} =\prod_{i=1}^d z_i^{\nu_i}$. We use $\mathrm{I}_d$ to denote the $d-$dimensional identity matrix. In the differentiation, for convenience, sometimes, we will use $h^{(s)}$ or $\partial^{s}h$ to denote $h^{(\boldsymbol{\nu})}$ for some $\boldsymbol{\nu}\in\mathbb{N}_0^d$ with $|\boldsymbol{\nu}|=s$, we use $\sum_{s\leq m}h^{(s)}$~(resp. $\sum_{s=m}h^{(s)}$) to denote $\sum_{\boldsymbol{\nu}\in \{\boldsymbol{\kappa}\in\mathbb{N}_0^d:|\boldsymbol{\kappa}|\leq m\}}h^{(\boldsymbol{\nu})}$~(resp. $\sum_{\boldsymbol{\nu}\in \{\boldsymbol{\kappa}\in\mathbb{N}_0^d:|\boldsymbol{\kappa}|= m\}}h^{(\boldsymbol{\nu})}$). We use $\norm{\cdot}$ to denote the Euclidean norm on Euclidean space. Unless otherwise specified, we use $C_{\cdot},C(\cdot),C$ to denote some generic constants that do not depend on other explicit parameters~($B_i$ or $K$ to be defined later).  {For function spaces, such as spaces of continuously differentiable functions, potential or Sobolev spaces, we use classical notations.}

		\section{Main Results}
        We first introduce the main theorem of this work, then we prove it using Galerkin method. In the end, we apply the main theorem to CBO dynamics.
        \subsection{Main theorem}
		Let $G(v,t),g(v,t)$ be scalar functions and $\J(v,t)$ be vector valued function, we study the following equations {in the time interval $[0,T]$, where $T<\infty$}:
		\begin{equation}\label{eq:g1}
			\begin{aligned}
&\partial_t\rho(v,t)=\div\Big(G(v,t)\nabla\rho(v,t)\Big)+\inner{\J(v,t)}{\nabla\rho(v,t)}+\rho(v,t)+g(v,t),\\
                &\rho(v,0)=\varrho(v)\in H^m(\mathbb{R}^d),\quad\forall m\geq 0,
			\end{aligned}
		\end{equation}
        and 
        \begin{equation}\label{eqg:2}
			\begin{aligned}
&\partial_t\rho(v,t)=\div\Big(G(v,t)\nabla\rho(v,t)\Big)-\div{\Big(\J(v,t)\rho(v,t)\Big)}+\rho(v,t)+g(v,t),\\
                &\rho(v,0)=\varrho(v)\in H^m(\mathbb{R}^d),\quad\forall m\geq 0.
			\end{aligned}
		\end{equation}
\rrevised{ We make a few remarks on the zeroth-order term $\rho(v,t)$ appearing on the right-hand sides of \eqref{eq:g1} and \eqref{eqg:2}. This term naturally arises later in the CBO dynamics~\eqref{eq:CBOrev}. Moreover, the arguments below remain valid if the zeroth-order term $\rho(v,t)$ is replaced by $h(v,t)\rho(v,t)$, where
\[
h\in L^\infty(0,T;W^{m,\infty}(\mathbb{R}^d)),
\]
and, for each $v\in\mathbb{R}^d$, the map $t\mapsto h(v,t)$ is locally H\"older continuous.
}

Let us first define a function class.
\begin{definition}The function class $\Gamma_G$ is defined by
    \begin{equation}
	\begin{aligned}
		\Gamma_G&:=\Big\{\phi:\mathbb R^d\times[0,T]\to \mathbb R: \text{for any } m\geq 0, \phi\in \mathcal{C}^{2,1}(\mathbb{R}^d\times [0,T])\cap  W^{1,\infty}(0,T;H^m(\mathbb{R}^d)), \\
        &\intd G^2(v,t)\norm{\phi^{(m)}(v,t)}^2dv<\infty, \forall t\in [0,T],\text{ and }\intt\intd G^3(v,t)\norm{\phi^{(m)}(v,t)}^2dvdt<\infty\Big\}.
	\end{aligned}
\end{equation}
Here, $\mathcal{C}^{k,\ell}$ represents the space of two variable functions that are $k$-times continuous differentiable in the first variable and $\ell$-times continuous differentiable with respect to the second. 
\end{definition}

\rrevised{
\begin{remark}
The weights $G^2$ and $G^3$ in the definition of $\Gamma_G$ come from the weighted energy estimates needed for \rrevised{uniqueness arguments as appearing in Sections~\ref{sec:304} and~\ref{sec:305}}. More precisely, in the improved estimates we use the test function
\[
    \psi=\partial^m(Q_i^2\rho_i^{(m)}),
\]
where $Q_i$ is comparable to $1+G_i$. This yields the bounds
\[
    \int_{B_i} Q_i^2\|\rho_i^{(m)}\|^2\,dv
    \qquad\text{and}\qquad
    \int_0^T\int_{B_i}G_iQ_i^2\|\nabla\rho_i^{(m)}\|^2\,dvdt .
\]
Since $Q_i\sim 1+G_i$, these estimates correspond to the $G^2$- and $G^3$-weighted bounds in $\Gamma_G$, respectively. The higher \rrevised{power} weights are used to control the commutator terms that arise when differentiating the coefficients $G$ and $J$, whose derivatives may grow like $1+G$ and $1+G^{1/2}$ under Assumption~\ref{asp:g1}. Thus a single $G$-weight is not sufficient to close the uniqueness estimates; see Sections~\ref{sec:304} and~\ref{sec:305}.
\end{remark}
}

The next is the definition of weak solution used in this paper.
\begin{definition}\label{def:sima24}
    We say $\rho$ is a weak solution to equation \eqref{eq:g1}, if $\rho(\cdot,t)$ is a Radon measure~(with respect to the Borel $\sigma$-algebra on $\mathbb{R}^d$) and for any test function $\psi\in \Gamma_G$, we have
    \begin{equation}
        \begin{aligned}     &\int_0^{s}\intd\Big(\partial_t\psi(v,t)+\div(G(v,t)\nabla\psi(v,t))-\div({\J(v,t)}{\psi(v,t)})+\psi(v,t)\Big)\rho(v,t)dvdt\\
            &=\intd\rho(v,s)\psi(v,s)dv-\intd\varrho(v)\psi(v,0)dv-\int_0^s\intd g(v,t)\psi(v,t)dvdt,
        \end{aligned}
    \end{equation}
    for any $s\in [0,T]$. We say $\rho$ is a weak solution to equation \eqref{eqg:2}, if $\rho(\cdot,t)$ is a Radon measure~(with respect to the Borel $\sigma$-algebra on $\mathbb{R}^d$) and for any test function $\psi\in \Gamma_G$, we have
    \begin{equation}
        \begin{aligned}
&\int_0^{s}\intd\Big(\partial_t\psi(v,t)+\div(G(v,t)\nabla\psi(v,t))+\inner{\J(v,t)}{\nabla\psi(v,t)})+\psi(v,t)\Big)\rho(v,t)dvdt\\
            &=\intd\rho(v,s)\psi(v,s)dv-\intd\varrho(v)\psi(v,0)dv-\int_0^s\intd g(v,t)\psi(v,t)dvdt,
        \end{aligned}
    \end{equation}
    for any $s\in [0,T]$.
\end{definition}

We proved the following the theorem.
\begin{theorem}[\textbf{Main Theorem}]\label{thm:main}
    If Assumptions \ref{asp:g1} and \ref{asp:g2} are satisfied, equation \eqref{eq:g1} and equation \eqref{eqg:2} have a classical solution in function class $\mathcal{C}^{2,1}(\mathbb{R}^d\times[0,T])\cap W^{1,\infty}(0,T;H^m_{loc}(\mathbb{R}^d))$ and satisfies the estimates
    	\begin{eqnarray}
    	&\sum_{s<m+1}\int_{\mathbb{R}^d}\norm{\rho^{(s)}}^2dv\leq C(T,d,m+1,\norm{g}_{L^{\infty}(0,T;H^{m}(\mathbb{R}^d))},\norm{\varrho}_{H^{m}(\mathbb{R}^d)})<\infty,\\
    	&\sum_{0<s\leq m+1}\intt\int_{\mathbb{R}^d} G\norm{\rho^{(s)}}^2dvdt\leq C(T,d,m+1,\norm{g}_{L^{\infty}(0,T;H^{m}(\mathbb{R}^d))},\norm{\varrho}_{H^{m}(\mathbb{R}^d)})<\infty,
    \end{eqnarray}
    for any $m\geq 0$. If Assumptions \ref{asp:g1}, \ref{asp:g2}, and \ref{asp:g3} are satisfied, equation \eqref{eq:g1} and equation \eqref{eqg:2} have a unique weak solution, and this weak solution is in function class $\Gamma_G$.
\end{theorem}
The proof is based on a Galerkin approximation; to perform such an approximation, we first introduce the assumptions in the next section.

\subsection{Assumptions }\label{sec:aspaddr}
In this section, we list the assumptions that $G,\J,g$ should satisfy.
		\begin{assumption}\label{asp:g1}
			\begin{itemize}
                    \item For each fixed $t\in [0,T]$, $J(\cdot,t),G(\cdot,t),g(\cdot,t)$ are smooth on $\mathbb{R}^d$ and for each fixed $v\in\mathbb{R}^d,$ $J(v,\cdot),G(v,\cdot),\nabla G(v,\cdot),\div{(J)}(v,\cdot),g(v,\cdot)$ are locally H\"older continuous in time with H\"older exponent $\alpha\in [0,1]$, $G\geq 0,g\in L^{\infty}(0,T;H^m(\mathbb{R}^d))$, for any $m\geq0$.
                    \item There exist uniform constants, which may depend on $T$, such that for each $v\in\mathbb{R}^d,t\in [0,T]$, we have
                    \begin{equation*}
                        \begin{aligned}
                            &\norm{G^{(1)}}\leq CG^{1/2}(1+G^{1/2});\\
                        &\norm{G^{(k)}}\leq C(1+G), \quad k\geq 2\text{ this $C$ may depend on $k$};\\
                        &\norm{\J}\leq CG^{1/2};\\
                        &\norm{\J^{(k)}}\leq C(1+G^{1/2}),\quad k\geq 1\text{ this $C$ may depend on $k$},
                        \end{aligned}
                    \end{equation*}
                    here the differentiation is only in the space variable.
			\end{itemize}
		\end{assumption}

		
		The next one is a technical assumption on $G$.
		\begin{assumption}\label{asp:g2}There exist two positive diverging sequences $\{n_i\}_{i=1}^{\infty}$ and $\{R_i\}_{i=1}^{\infty}$, such that $G(v,t)\leq n_i$ for any $t\in [0,T],v\in B_{R_i}(0)$; and for any $v\in B_{R_i}(0)\setminus B_{R_i-1}(0)$, we have $G(v,t)\leq C_1 \Big(1+G(\frac{R_i v}{\norm{v}},t)\Big), \Big(1+G(\frac{R_i v}{\norm{v}},t)\Big)\leq C_2G(v,t)$, here constants are uniform for $t\in [0,T]$ and do not depend on $i$.
		\end{assumption}

\begin{remark}
    Following the proof {of Theorem \ref{thm:main}}, it is not hard to see the main theorem also holds when $G(v,t)$ is a {symmetric} square matrix. In this case, the assumptions should be modified in a matrix form: \begin{equation*}
                        \begin{aligned}
                            &- G\text{ is non-negative definite };\\
                            &-CG^{1/2}(1+G^{1/2})\leq G^{(1)}\leq CG^{1/2}(1+G^{1/2});\\
                        &-C(1+G)\leq G^{(k)}\leq C(1+G), \quad k\geq 2\text{ this $C$ may depend on $k$};\\
                        & - \J\J^\top\leq CG;\\
                        & - (\J^{(k)})(\J^{(k)})^\top\leq C(1+G),\quad k\geq 1\text{ this $C$ may depend on $k$},
                        \end{aligned}
                    \end{equation*}
    here matrix inequality $A\leq B$ means $v^\top Av\leq v^\top Bv$, for any $v\in\mathbb{R}^d$. The condition $G(v,t)\leq n_i$ in Assumption \ref{asp:g2} should be adapted as $G(v,t)\leq n_i\mathrm{I}_d$.
\end{remark}

In  such assumptions, the constants are all uniform with respect to $t\in [0,T]$. So, for simplicity, in the next, we will often omit $t$ in $G(v,t),\J(v,t)$ and just write them as $G(v),\J(v)$.
		
\section{Proof of the main theorem by Galerkin approximation}
        {In this section, we concentrate on equation \eqref{eq:g1}, while the discussion of equation \eqref{eqg:2} will be included in remarks.}
        We first find solutions on tori, then use \rrevised{them}  to approximate the solutions on $\mathbb{R}^d$.

        Denote $B_i:=[-n_i^2,n_i^2]^d$, \rrevised{here $n_i$ is the sequence from Assumption \ref{asp:g2},} and
        \rrevised{\begin{equation}
        	\rrevised{\mathcal{H}_{i,K}}:=\Big\{\text{the constant function}=1,v\mapsto\cos(\frac{2\pi \boldsymbol{k}\cdot v}{{2n_i^2}}),v\mapsto\sin(\frac{2\pi \boldsymbol{k}\cdot v}{{2n_i^2}}):0<|\boldsymbol{k}|\leq K\Big\},
        \end{equation}}
        here $\boldsymbol{k}\in\mathbb{N}_0^d$, we can think of $B_i$ as a $d$-dimensional torus. Let us emphasize that with this choice of $B_i$, functions $G_i,\J_i$ and $g_i,\varrho_i$~(which will be defined later) equal $0$ near the boundary of $B_i$, which allows us to use integration by parts without considering the boundary terms.
        
        {In the next, we will \rrevised{provide} the definition and properties of $G_i,\J_i$. The definition of $g_i,\varrho_i$ will be given in the next subsection.
        
         Let 
        \begin{equation}
        	\varphi(x):= \begin{cases}e^{-1 /\left(1-|x|^2\right)} / I & \text { if }|x|<1, \\ 0 & \text { if }|x| \geq 1,\end{cases}
        \end{equation}
        here $x\in\mathbb{R}$ and $I$ is the \rrevised{$L^1-$}normalization constant, and $\varphi_\epsilon(x):=C(\epsilon)\varphi(\frac{x}{\epsilon})$, here $C(\epsilon)$ is the \rrevised{$L^1-$}normalization constant.   We define 
         \begin{equation}
        	S(x):=\Big(\chi_{[\frac{1}{2},\infty)}*\varphi_{\frac{1}{8}}\Big)(x),
        \end{equation}
        where $\chi$ is the characteristic function, $*$ represents convolution and
        \begin{eqnarray}
        	&S_i(v):=S(\norm{v}-R_i+1),\\
        	&\overline{G}_i(v):=G(v)(1-S_i(v))+\Big(1+G\left(\frac{R_iv}{\norm{v}} \right )\Big)S_i(v),\\
        	&\overline{\J}_i(v):=\J(v)(1-S_i(v))+\left (\sqrt{G\left(\frac{R_iv}{\norm{v}}\right )+1} S_i(v)\right) \e,
        \end{eqnarray}
        here $e:=(1,\ldots,1)^{\top}\in \mathbb{R}^d$} \rrevised{and $R_i$ is the sequence from Assumption \ref{asp:g2}}. {Then we have the following lemma.
        \begin{lemma}\label{lem:g25}Under Assumptions \ref{asp:g1} and \ref{asp:g2}, we have for any $v\in\mathbb{R}^d$ that
        	\begin{eqnarray}
        		&\norm{\overline{G}_i^{(1)}}\leq C\overline{G}^{\frac{1}{2}}_i\Big(1+\overline{G}^{\frac{1}{2}}_i\Big);\\
        		&\norm{\overline{G}_i^{(k)}}\leq C\Big(1+{\overline{G}_i}\Big),\quad \forall k\geq 2;\\
        		&\norm{\overline{\J}_i}^2\leq C\overline{G}_i;\\
        		&\norm{\overline{\J}_i^{(k)}}\leq C(1+\overline{G}^{\frac{1}{2}}_i)\quad \forall k\geq 1,
        	\end{eqnarray}
        	where the constants are uniform for any $t\in [0,T]$ and may depend on $k$ but do not depend on $i$.
        \end{lemma}
Next, we define 
     \begin{equation}
     	\begin{aligned}
     		&\H(x)=(\chi_{[-10,10]}*\varphi)(x),\\
     		&\H_i(x)=\H\left (\frac{x}{n_i} \right),
     	\end{aligned}
     \end{equation}
     then with this $H_i$, we can finally define
     \begin{eqnarray}
     	G_i(v)&:=H^2_i(\norm{v})\overline{G}_i(v),\label{eq:sima29}\\
     	\J_i(v)&:=H_i(\norm{v})\overline{\J}_i(v).\label{eq:sima30}
     \end{eqnarray}
The above defined $G_i$ and $J_i$, which have support within $B_{12n_i}(0)$~(provided $n_i$ is large enough, such that $12n_i\ll n_i^2$), converge pointwise to $G,J$ as $i\to\infty$, and they satisfy the following key lemma.
     \begin{lemma}\label{lem:g4}Under Assumptions \ref{asp:g1} and \ref{asp:g2}, we have for any $v\in\mathbb{R}^d,t\in [0,T]$ that
     	\begin{eqnarray}
     		&\norm{{G}_i^{(1)}}\leq C{G}^{\frac{1}{2}}_i\Big(1+{G}^{\frac{1}{2}}_i\Big);\\
     		&\norm{{G}_i^{(k)}}\leq C\Big(1+{{G}_i}\Big),\quad \forall k\geq 2;\\
     		&\norm{{\J}_i}^2\leq C{G}_i;\\
     		&\norm{{\J}_i^{(k)}}\leq C(1+{G}^{\frac{1}{2}}_i)\quad \forall k\geq 1,
     	\end{eqnarray}
     	here constants are uniform for any $t\in [0,T]$ and may depend on $k$ but do not depend on $i$.
     \end{lemma}
        
      We have included the proof of the two lemmas above in Appendix \ref{apx:2},
}\rrevised{see Lemma \ref{lem:gg25} and Lemma \ref{lem:gg4} therein.}

        \subsection{Finite dimensional solutions on tori}
		\renewcommand{\intr}{\int_{B_i}}

		We define function
		\begin{equation}
			\rho_i^K(v,t):=\sum_{k=1}^{K'}C^K_{i,k}(t)\psi_{i,k}(v),
		\end{equation}
		$\psi_{i,k}$ is from $\rrevised{\mathcal{H}_{i,K}}$, and $K'$ is the cardinality of $\rrevised{\mathcal{H}_{i,K}}$. Now, we want for any test function $\psi\in \operatorname{span}(\rrevised{\mathcal{H}_{i,K}})$ and any $t\in [0,T]$ that $\rho_i^K$ satisfies the following equations: 
		\begin{equation}\label{eq:gg55}\footnotesize
			\begin{aligned}
				&\intr \partial_t\rho_i^K(v,t)\psi(v)dv=\intr \Big(\div\Big(G_i(v,t)\nabla\rho^K_i(v,t)\Big)+\inner{\J_i(v,t)}{\nabla\rho^K_i(v,t)}+\rho^K_i(v,t)+g_{i}(v,t)\Big)\psi(v)dv,\\
				&\rho_i^K(v,0)=\varrho^K_i(v),
			\end{aligned}
		\end{equation}
		here $G_i,\J_i$ are from \eqref{eq:sima29},\eqref{eq:sima30}, $g_i(v,t):=g(v,t)(1-S(\norm{v}-n_i))$ and $\varrho_i^K(v)$ is the projection of $\varrho(v)(1-S(\norm{v}-n_i))$ onto $\operatorname{span}(\rrevised{\mathcal{H}_{i,K}})$. This requires the coefficient $C^K_{i,k}(t)$ to satisfy the following equations:
		\begin{equation}\label{eq:89}
			\begin{aligned}
				&\sum_{k=1}^{K'} A^K_{i,k,j}\frac{d}{dt}C^K_{i,k}(t)= \sum_{k=1}^{K'}B^K_{i,k,j}(t)C_{i,k}^K(t)+g^K_{i,j}(t)\\
				&C_{i,j}^K(0)=I^K_{i,j},
			\end{aligned}
		\end{equation}
		for $i=1,\ldots, K$, where
		\begin{align}
			&A^K_{i,k,j}=\intr \psi_{i,k}(v)\psi_{i,j}(v)dv=\rrevised{D_{ik}}\delta_{kj},\quad \rrevised{D_{ik}}>0,\\
			&B^K_{i,k,j}(t)=-\intr \inner{{G_i(v,t)}\nabla\psi_{i,k}(v))}{\nabla\psi_{i,j}(v)}dv+\intr \inner{\J_i(v,t)}{\nabla\psi_{i,k}(v)}\psi_{i,j}(v)dv\\
			&\qquad\qquad\quad+\intr \psi_{i,k}(v)\psi_{i,j}(v)dv,\\
			&g^K_{i,j}(t):=\intr g_i(v,t)\psi_{i,j}(v)dv,\\
			&I_{i,j}^K:=\intr \varrho^K_i(v)\psi_{i,j}(v)dv.
		\end{align}
	{
    \begin{remark}\label{rem:29}
	    By \cite[Theorem 2.1.]{shen2011spectral}, we have 
        \begin{equation*}
            \norm{\varrho_i^K-\varrho(v)(1-S(\norm{v}-n_i))}_{H^{l_1}(B_i)}\lesssim K^{l_1-l_2}\norm{\varrho(v)(1-S(\norm{v}-n_i))}_{H^{l_2}(B_i)}.
        \end{equation*}
        So if $\norm{\varrho}_{H^m(\mathbb{R}^d)}<\infty$, then $\norm{\varrho(v)(1-S(\norm{v}-n_i))}_{H^{m}(B_i)}\leq C(\norm{\varrho}_{H^m(\mathbb{R}^d)},m)<\infty$ and $\norm{\varrho_i^K}_{H^m(B_i)}\lesssim C(\norm{\varrho}_{H^m(\mathbb{R}^d)},m)<\infty$, which does not depend on $K$.
	\end{remark}
    }
		We know by ODE theory that the linear system \eqref{eq:89} has a unique solution till some time $T'>0$, so
		$\rho_i^K$ is well-defined for $t$ in $[0,T')$. And for any $\psi$ such that $\psi(\cdot,t)\in \mathrm{span}(\rrevised{\mathcal{H}_{i,K}})$, we have for any fixed $t\in [0,T')$ that
		\begin{equation}\label{eq:94}
			\begin{aligned}
				\intr \partial_t\rho_i^K(v,t) \psi(v,t)dv&=-\intr \inner{{G_i(v,t)}\nabla\rho_i^K(v,t))}{\nabla\psi(v,t)}dv\\
				&\quad+\intr \inner{\J_i(v,t)}{\rrevised{\nabla} \rho_i^K(v,t)}\psi(v,t)dv\\
				&\quad +\intr\rho_i^K(v,t)\psi(v,t)dv+\intr g_i(v,t)\psi(v,t)dv.
			\end{aligned}
		\end{equation}
		
		In rest of this subsection, for simplicity, we will omit $K,i$ in $\rho_i^K$. Choose $\psi=\rho~(=\rho_i^K)\in \operatorname{span}(\rrevised{\mathcal{H}_{i,K}})$ {in \eqref{eq:gg55}} \rrevised{or \eqref{eq:94}}, we have
		\begin{equation}\label{eq:95}
			\begin{aligned}
				\frac{1}{2}\frac{d}{dt}\intr\rho^2dv&=-\intr\inner{G_i\nabla\rho}{\nabla\rho} dv+\intr\inner{\J_i}{\nabla\rho}\rho dv+\intr \rho^2 dv+\intr g_i\rho dv\\
				&\leq -\intr G_i\norm{\nabla\rho}^2 dv+\epsilon\intr \norm{\J_i}^2\norm{\nabla\rho}^2dv+\Big(\frac{C}{\epsilon}+2\Big)\intr\rho^2dv+\intr g_i^2dv\\
				&\leq -(1-C\epsilon)\intr{G_i}\norm{\nabla\rho}^2dv+C(\epsilon)\intr\rho^2dv+\norm{g(\cdot,t)}^2_{L^2(\mathbb{R}^d)},
			\end{aligned}
		\end{equation}
		where in the last inequality, we used Lemma \ref{lem:g4}. \rrevised{Then we choose $\epsilon$ small enough to obtain}
		{\begin{equation}\label{eq:sima45}
			\frac{d}{dt}\intr\rho^2dv\leq C\Big(\intr\rho^2dv+\norm{g(\cdot,t)}^2_{L^2(\mathbb{R}^d)}\Big).
		\end{equation}}
		\rrevised{Going back to \eqref{eq:95}, we
integrate it and rearrange the terms to obtain} 
		\begin{equation}\footnotesize\label{eq:sima46}
		\int_0^{T'}\intr{G_i}\norm{\nabla\rho}^2dvdt\leq C\Big(\intr\rho^2(v,0)dv-\intr\rho^2(v,T')dv+\int_0^{T'}\intr\rho^2dvdt+\int_0^{T'}\norm{g(\cdot,t)}_{L^2(\mathbb{R}^d)}dt\Big).
		\end{equation}
Then by Gr\"onwall lemma, for any $t\in [0,T')$, we have
		\begin{align}
			&\intr\rho^2(v,t)dv\leq C(\norm{g}_{L^\infty(0,T';L^2(\mathbb{R}^d))},\norm{\varrho}_{L^2(\mathbb{R}^d)},T')<\infty,\label{eq:sima79}\\
			&\int_0^{T'}\intr G_i\norm{\nabla\rho}^2dvdt\leq C(T',\norm{g}_{L^\infty(0,T';\rrevised{L^2(\mathbb{R}^d)})},\norm{\varrho}_{L^2(\mathbb{R}^d)})<\infty,\label{sima:79'}
		\end{align}
		{the constants on the right hand sides are independent of $i,K$, and depend on $T'$ exponentially. Refer also to Remark \ref{rem:29}. Because of the bound \eqref{eq:sima79}, we can extend $T'$ to $T$, obtaining a global solution over $[0,T]$.}
		
		Next, for $m\geq 1$, we assume
		\begin{equation}\label{eqs:g66}
			\begin{aligned}
					&\sum_{s<m}\intr\norm{\rho^{(s)}}^2dv\leq C(T,d,m,\norm{g}_{L^{\infty}(0,T;H^{m-1}(\mathbb{R}^d))},\norm{\varrho}_{H^{m-1}(\mathbb{R}^d)})<\infty,\\
					&\sum_{0<s\leq m}\intt\intr G_i\norm{\rho^{(s)}}^2dvdt\leq C(T,d,m,\norm{g}_{L^{\infty}(0,T;H^{m-1}(\mathbb{R}^d))},\norm{\varrho}_{H^{m-1}(\mathbb{R}^d)})<\infty.
			\end{aligned}
		\end{equation}
        Then under this assumption, we will prove the above inequalities are also true for $m+1$ as induction step. When $m=1$, we have proved already that the above estimates are true, as in \eqref{eq:sima79}-\eqref{sima:79'}.

		For $m\geq 1$, we choose $\psi=\partial^m\partial^m\rho=:(\partial^m)^2\rho$ as the test function. We know for any $t$, $\psi(\cdot,t)\in \mathrm{span}(\rrevised{\mathcal{H}_{i,K}})$. Then by periodicity, integration by parts and compensation of signs, we have
		\begin{equation}\label{eq:99}
			\begin{aligned}
				\frac{1}{2}\frac{d}{dt}\intr\norm{\partial^m\rho}^2dv&=-\intr\inner{\nabla\rho^{(m)}}{\partial^m(G_i\nabla\rho)}dv\\
				&\quad+\intr\partial^m\Big(\inner{\J_i}{\nabla\rho}\Big)\rho^{(m)}dv\\
				&\quad+\intr\norm{\rho^{(m)}}^2dv\\
				&\quad+\intr g^{(m)}_i\rho^{(m)} dv,
			\end{aligned}
		\end{equation}
      thus
        \begin{equation}\label{eq:sima83}
            \begin{aligned}
    \frac{1}{2}\frac{d}{dt}\intr\norm{\partial^m\rho}^2dv&\leq -\intr G_i\norm{\nabla\rho^{(m)}}^2 dv-\intr\inner{\J_i}{\nabla\rho^{(m)}}\rho^{(m)}dv\\
&\quad+\intr\norm{\rho^{(m)}}^2dv+\intr g_i^{(m)}\rho^{(m)} dv\\
&\quad +C\norm{\intr\inner{\nabla\rho^{(m)}}{G_i^{(1)}\nabla\rho^{(m-1)}}dv}\\
&\quad+C\sum_{s=2}^m\norm{\intr \inner{\nabla\rho^{(m)}}{G_i^{(s)}\nabla\rho^{(m-s)}}dv}\\
&\quad +C\sum_{s=1}^m\norm{\intr\inner{\J_i^{(s)}}{\nabla\rho^{(m-s)}}\rho^{(m)}dv},
            \end{aligned}
        \end{equation}
		note when $m=1$, the sixth term in the right hand side is zero. The summation of the first four terms on the right hand side, just like in the case where $m=0$, can be bounded from above by 
		\begin{equation}
			-(1-C\epsilon)\intr G_i\norm{\nabla\rho^{(m)}}^2dv+C(\epsilon)\intr \norm{\rho^{(m)}}^2dv+C\norm{g(\cdot,t)}_{H^m(\mathbb{R}^d)};
		\end{equation}
		Using Lemma \ref{lem:g4} and Cauchy-\rrevised{Schwarz} inequality, the fifth term is less than
		\begin{equation}
			\epsilon\intr G_i\norm{\nabla\rho^{(m)}}^2dv+C(\epsilon)\Big(\intr\norm{\nabla \rho^{(m-1)}}^2dv+\intr G_i\norm{\nabla\rho^{(m-1)}}^2dv\Big);
		\end{equation}
		when $m\geq 2$, the sixth term,  using integration by parts and Lemma \ref{lem:g4}, is less than
		\begin{equation}
			\begin{aligned}
				&\sum_{s=2}^m\norm{\intr \inner{\nabla\rho^{(m)}}{G_i^{(s)}\nabla\rho^{(m-s)}}dv}\\
				&\leq \sum_{s=2}^m\norm{\intr \rho^{(m)}\inner{\nabla G_i^{(s)}}{\nabla\rho^{(m-s)}}dv}+\sum_{s=2}^m\norm{\intr G_i^{(s)}\rho^{(m)}\Delta\rho^{(m-s)}dv}\\
                &\leq C\left\{\sum_{s=2}^m{\intr \Big(1+G_i\Big)\norm{\rho^{(m)}}\norm{\nabla\rho^{(m-s)}}dv}+\sum_{s=2}^m{\intr \Big(1+G_i\Big)\norm{\rho^{(m)}}\norm{\Delta\rho^{(m-s)}}dv}\right\}\\
				&\leq C\Big\{\intr\norm{\rho^{(m)}}^2dv+\sum_{s=2}^{m}\int\norm{\nabla \rho^{(m-s)}}^2dv+\intr G_i\norm{\rho^{(m)}}^2dv+\sum_{s=2}^m\int G_i\norm{\nabla \rho^{(m-s)}}^2dv\\
				&\quad +\sum_{s=2}^{m}\int\norm{\Delta \rho^{(m-s)}}^2dv+\sum_{s=2}^m\int G_i\norm{\Delta \rho^{(m-s)}}^2dv\Big\}\\
				&\leq C\sum_{1\leq s\leq m}\Big(\intr \norm{\rho^{(s)}}^2dv+\intr G_i\norm{\rho^{(s)}}^2dv\Big);
			\end{aligned}
		\end{equation}
		whereas the last term,  by using Lemma \ref{lem:g4}, is less than
		\begin{equation}
			\begin{aligned}
				\sum_{s=1}^m\norm{\intr\inner{\J_i^{(s)}}{\nabla\rho^{(m-s)}}\rho^{(m)}dv}
                &\leq C\sum_{s=1}^m\intr \Big(1+G_i^{\frac{1}{2}}\Big)\norm{\rho^{(m)}}\norm{\nabla\rho^{(m-s)}}dv \\
                &\leq C\sum_{1\leq s\leq m}\Big(\intr \norm{\rho^{(s)}}^2dv+\intr G_i\norm{\rho^{(s)}}^2dv\Big).
			\end{aligned}
		\end{equation}
		So, by combining all the estimates, {integrating both sides from $0$ to $t$}, and using condition \eqref{eqs:g66}, we obtain
	\begin{equation}\label{eq:g88}
		\begin{aligned}
			&\sum_{s=m}\intr\norm{\rho^{(s)}(v,t)}^2dv-\sum_{s=m}\intr\norm{\rho^{(s)}(v,0)}^2dv\\
			&\leq -(1-C\epsilon)\sum_{s=m}\int_0^t\intr G_i(v,r)\norm{\nabla\rho^{(s)}(v,r)}^2dvdr\\
			&\quad+C\sum_{s=m}\int_0^t\intr\norm{\rho^{(s)}(v,r)}^2dvdr\\
			&\quad+C(T,d,m,\epsilon,\norm{g}_{L^{\infty}(0,T;H^{m}(\mathbb{R}^d))},\norm{\varrho}_{H^{m-1}(\mathbb{R}^d)}).
		\end{aligned}
	\end{equation}
    {\rrevised{Let us now recall that, in} the first line of \eqref{eq:g88}, $$\sum_{s=m}\intr\norm{\rho^{(s)}(v,t)}^2dv:=\sum_{\boldsymbol{\nu}\in \{\boldsymbol{\kappa}\in\mathbb{N}_0^d:|\boldsymbol{\kappa}|= m\}}\intr \norm{\rho^{(\boldsymbol{\nu})}(v,t)}^2dv$$ as in the Notation section.}
	Then by choosing $\epsilon$ small enough, and using Gr\"onwall lemma, we have
	\begin{eqnarray}
		&\sum_{s<m+1}\intr\norm{\rho^{(s)}}^2dv\leq C(T,d,m+1,\norm{g}_{L^{\infty}(0,T;H^{m}(\mathbb{R}^d))},\norm{\varrho}_{H^{m}(\mathbb{R}^d)})<\infty,\label{eq:g89}\\
		&\sum_{0<s\leq m+1}\intt\intr G_i\norm{\rho^{(s)}}^2dvdt\leq C(T,d,m+1,\norm{g}_{L^{\infty}(0,T;H^{m}(\mathbb{R}^d))},\norm{\varrho}_{H^{m}(\mathbb{R}^d)})<\infty,\label{eq:g90}
	\end{eqnarray}
note the above constants \rrevised{do} not depend on $K,i$. The derivation from \eqref{eq:g88} to \eqref{eq:g89}-\eqref{eq:g90} follows a similar argument as before, see \eqref{eq:sima45}-\eqref{eq:sima46}.

	Next, consider the test function $\psi(v,t)=(\partial^{m})^2\partial_t\rho(v,t)$, then use \eqref{eq:gg55}, integration by parts and observe that $\intr\norm{\partial^m\Big(\div(G_i\nabla\rho)\Big)}^2dv\leq C(n_i)\norm{\rho}^2_{L^{\infty}(0,T;H^{m+2}(B_i))}$, as $G_i\leq n_i+1$ \rrevised{and we can use Lemma \ref{lem:g4}}, to obtain
		\begin{equation}\label{eq:g85}
			\begin{aligned}
				\intr\norm{\partial_t\partial^m\rho}^2dv&=-\intr\inner{\partial^m\Big(\div(G_i\nabla\rho)\Big)}{\partial_t\rho^{(m)}}dv\\
				&\quad+\intr\partial^m\Big(\inner{\J_i}{\rho}\Big)\partial_t\rho^{(m)}dv\\
				&\quad+\intr\rho^{(m)}\partial_t\rho^{(m)}dv+\intr g_i^{(m)}\partial_t\rho^{(m)}dv\\
				&\leq \epsilon\intr\norm{\partial_t\partial^m\rho}^2dv+\frac{C(n_i)}{\epsilon}\Big(\norm{\rho}^2_{L^{\infty}(0,T;H^{m+2}(B_i))}+\norm{g}^2_{L^{\infty}(0,T;H^m(\mathbb{R}^d))}\Big).
			\end{aligned}
		\end{equation}
	 {So by the estimates of \eqref{eq:g89}}, we have
		\begin{equation}\label{eq:gg106}
			\begin{aligned}
			    \intr\norm{\partial_t\partial^m\rho}^2dv&\leq \frac{C(n_i)}{\epsilon(1-\epsilon)}\Big(\norm{\rho}^2_{L^{\infty}(0,T;H^{m+2}(B_i))}+\norm{g}^2_{L^{\infty}(0,T;H^{m}(\mathbb{R}^d))}\Big)\\
                &{\leq C(n_i)C(T,d,m+2,\epsilon,\norm{g}_{L^{\infty}(0,T;H^{m+2}(\mathbb{R}^d))},\norm{\varrho}_{H^{m+2}(\mathbb{R}^d)})}\\
                &<\infty,
			\end{aligned}
		\end{equation}
		{which does not depend on $K$.} Thus we have $\rho^K_i(=:\rho)\in W^{1,\infty}(0,T;H^{m}(B_i))$, for any $m\geq 0$, and its norm does not depend on $K$, but {linearly depends on $C(n_i)$} in this case.

\subsection{Limit solutions on tori by letting $K\to\infty$}
	In this subsection, we show $\rho_i^K$ has a limit as $K\to\infty$. {By \eqref{eq:sima79} and the properties of $\psi_{i,k}$},	we first have
		\begin{equation}
			\intr (\rho^K_i)^2dv=\sum_{k=1}^{K}(C_{i,k}^K(t))^2\intr(\psi_{i,k})^2dv\leq C(T,\norm{g}_{L^{\infty}(0,T;L^2(\mathbb{R}^d))},\norm{\varrho}_{L^2(\mathbb{R}^d)})<\infty,
		\end{equation}
		which does not depend on $i,K$, and thus we have
		\begin{equation}\label{eq:sima94}
			\norm{C_{i,k}^K(t)}^2\leq\frac{C(T,\norm{g}_{L^{\infty}(0,T;L^2(\mathbb{R}^d))},\norm{\varrho}_{L^2(\mathbb{R}^d)})}{\intr(\psi_{i,k})^2dv}<\infty.
		\end{equation}
		{The right hand sides of \eqref{eq:sima94} depends on $\psi_{i,k}$, but not on $K$}. So, for a fixed $k$, $C_{i,k}^K(t)$ is uniformly bounded on $[0,T]$ when $K\to\infty$. Second, we will show $C_{i,k}^K(t)$ is uniformly Lipschitz continuous, i.e., the Lipschitz constant does not depend on $K$. 
Choose $\psi=\psi_{i,k}$ in equation \eqref{eq:gg55}, we have
		\begin{equation}
			\begin{aligned}
				\norm{\intr (\psi_{i,k})^2dv\frac{d}{dt}C^K_{i,k}(t)}&\leq \norm{\intr \inner{G_i(v,t)\nabla\rho^K_i(v,t)}{\nabla\psi_{i,k}(v)}dv}\\
				&\quad+\norm{\intr \inner{\J_i(v,t)}{\nabla\rho^K_i(v,t)}\psi_{i,k}(v)dv}\\
				&\quad +\norm{\intr \rho^K_i(v,t)\psi_{i,k}(v)dv}+\norm{\intr g_i(v,t)\psi_{i,k}(v)dv}\\
				&\leq C(n_i,T,d,\norm{g}_{L^{\infty}(0,T;L^2(\mathbb{R}^d))},\norm{\varrho}_{L^{\infty}(0,T;H^1(\mathbb{R}^d))})\norm{\psi_{i,k}}_{H^1(B_i)},
			\end{aligned}
		\end{equation}
		as $G_i\leq n_i+1$ {and the estimates of \eqref{eq:g89}, \eqref{eq:g90}, and \eqref{eq:gg106}} hold; thus
		\begin{equation}
			\norm{\frac{d}{dt}C^K_{i,k}(t)}\leq \frac{C(n_i,T,d,\norm{g}_{L^{\infty}(0,T;L^2(\mathbb{R}^d))},\norm{\varrho}_{L^{\infty}(0,T;H^1(\mathbb{R}^d))})\norm{\psi_{i,k}}_{H^1(B_i)}}{\norm{\psi_{i,k}}^2_{L^2(B_i)}},
		\end{equation}
		which {depends on $k,i$ but} is independent of $K$. Hence \rrevised{for a fixed $k$}, $C^K_{i,k}(t)$ is uniformly \rrevised{Lipschitz continuous} on $[0,T]$ as $K\to\infty$.
        \rrevised{By the Arzel\`a--Ascoli theorem, we may construct a diagonal subsequence as follows. 
For \(k=1\), there exists a subsequence \(\{a_{1j}\}_{j=1}^{\infty}\subset \mathbb N\) such that
\[
C_{i,1}^{a_{1j}}(t)
\]
converges uniformly on \([0,T]\) to a Lipschitz continuous function, which we denote by \(C_{i,1}(t)\).

Next, for \(k=2\), we choose a subsequence \(\{a_{2j}\}_{j=1}^{\infty}\) of \(\{a_{1j}\}_{j=1}^{\infty}\) such that
\[
C_{i,2}^{a_{2j}}(t)
\]
converges uniformly on \([0,T]\) to a Lipschitz continuous function, denoted by \(C_{i,2}(t)\).

Proceeding inductively, suppose that the subsequence \(\{a_{nj}\}_{j=1}^{\infty}\) has been chosen. For \(k=n+1\), we select a subsequence \(\{a_{(n+1)j}\}_{j=1}^{\infty}\) of \(\{a_{nj}\}_{j=1}^{\infty}\) such that
\[
C_{i,n+1}^{a_{(n+1)j}}(t)
\]
converges uniformly on \([0,T]\) to a Lipschitz continuous function, denoted by \(C_{i,n+1}(t)\).

Finally, we define the diagonal subsequence
\[
K_j:=a_{jj}.
\]
Hence, by construction, we have 
\[
C_{i,k}^{K_j}(t)\longrightarrow C_{i,k}(t)
\quad\text{uniformly on }[0,T],\quad \text{as } j\to\infty,
\]
for every fixed \(k\).        
}
 \rrevised{Accordingly, we set
		\begin{equation}
			\rho_i(v,t):=\sum_{k=1}^\infty C_{i,k}(t)\psi_{i,k}(v),
		\end{equation}
        which is well-defined because of the uniform bound \eqref{eq:sima79}
	   and by Fatou's lemma; in particular we have $\rho_i\in L^{\infty}(0,T;L^2(B_i))$.
        For this limit function $\rho_i(v,t)$, we show} the estimates of \eqref{eq:g89}, \eqref{eq:g90}, and \eqref{eq:gg106} also hold.

\rrevised{

Since the sequence \(\{\rho_i^{K_j}\}_{j=1}^{\infty}\) is uniformly bounded in
\[
W^{1,\infty}(0,T;H^m(B_i)),
\]
with a bound $C$ independent of \(K_j\), the Banach--Alaoglu theorem~(see the predual of $W^{1,\infty}(0,T;H^m(B_i))$ in Appendix \ref{sec:k6}) implies that, up to extracting a further subsequence, still denoted by \(\{K_j\}_{j=1}^{\infty}\), there exists
\[
\rho_i'\in W^{1,\infty}(0,T;H^m(B_i))
\]
such that
\[
\rho_i^{K_j}\rightharpoonup^\ast \rho_i'
\quad\text{in } W^{1,\infty}(0,T;H^m(B_i)).
\]
Moreover,
\[
\|\rho_i'\|_{W^{1,\infty}(0,T;H^m(B_i))}
\le C.
\]

On the other hand, along the same subsequence, the convergence of the
coefficients used to define \(\rho_i\) implies that, for every fixed
\(t\in[0,T]\),
\[
\rho_i^{K_j}(\cdot,t)\rightharpoonup \rho_i(\cdot,t)
\quad\text{weakly in }L^2(B_i).
\]
Indeed, for each basis function \(\psi_{i,k}\), we have
\[
\int_{B_i}\rho_i^{K_j}(v,t)\psi_{i,k}(v)\,dv
=
C_{i,k}^{K_j}(t)
\longrightarrow
C_{i,k}(t)
=
\int_{B_i}\rho_i(v,t)\psi_{i,k}(v)\,dv .
\]
Since the linear span of \(\{\psi_{i,k}\}_{k\ge1}\) is dense in
\(L^2(B_i)\), and the sequence \(\rho_i^{K_j}(\cdot,t)\) is uniformly
bounded in \(L^2(B_i)\), this coefficient convergence extends to all test
functions in \(L^2(B_i)\), giving the claimed weak convergence.
Therefore, the weak-\(*\) limit \(\rho_i'\) coincides with \(\rho_i\). Consequently,
\[
\rho_i\in W^{1,\infty}(0,T;H^m(B_i)),
\qquad
\|\rho_i\|_{W^{1,\infty}(0,T;H^m(B_i))}
\le C.
\]
This verifies the estimates \eqref{eq:g89} and \eqref{eq:gg106}.

Next, we invoke the following compactness lemma, which may be viewed as a variant of the Aubin--Lions lemma. For completeness, its proof is postponed to Appendix~\ref{apx:n210}.
\begin{lemma}\label{lem:n210}
Let \(\Omega\subset\mathbb R^d\) be a bounded smooth domain, and let \(m\ge 0\) be an integer. 
Assume that \(M>m+\frac d2\). Suppose that
\[
\rho_j \rightharpoonup^\ast \rho
\quad\text{in } W^{1,\infty}(0,T;H^M(\Omega)).
\]
Then
\[
\rho_j\to \rho
\quad\text{strongly in } \mathcal{C}(0,T;C^m(\Omega)).
\]
\end{lemma}

We apply the above lemma to $\rho_i^{K_j},\rho_i$, then  Fatou's lemma yields
\begin{equation}
\begin{aligned}
\sum_{0<s\leq m+1}
\int_0^T\int_{B_i}
G_i\|\rho_i^{(s)}\|^2\,dv\,dt
&\le
\liminf_{j\to\infty}
\sum_{0<s\leq m+1}
\int_0^T\int_{B_i}
G_i\|(\rho_i^{K_j})^{(s)}\|^2\,dv\,dt  \\
&\le
C\Big(
T,d,m+1,
\|g\|_{L^\infty(0,T;H^m(\mathbb R^d))},
\|\varrho\|_{H^m(\mathbb R^d)}
\Big)
<\infty .
\end{aligned}
\end{equation}
Thus the estimate \eqref{eq:g90} is also verified.

}
\rrevised{
By the weak-\(*\) convergence in \(W^{1,\infty}(0,T;H^m(B_i))\), we may pass to the limit in the Galerkin formulation. Hence, for every
\[
\psi\in L^1(0,T;H^1(B_i)),
\]
we obtain
\begin{equation}\label{eq:t94}
\begin{aligned}
\int_0^T\int_{B_i} \partial_t\rho_i(v,t)\psi(v,t)\,dv\,dt
&=
-\int_0^T\int_{B_i}
\left\langle G_i(v,t)\nabla\rho_i(v,t),\nabla\psi(v,t)\right\rangle\,dv\,dt \\
&\quad
+\int_0^T\int_{B_i}
\left\langle J_i(v,t),\nabla\rho_i(v,t)\right\rangle
\psi(v,t)\,dv\,dt \\
&\quad
+\int_0^T\int_{B_i}
\rho_i(v,t)\psi(v,t)\,dv\,dt
+
\int_0^T\int_{B_i}
g_i(v,t)\psi(v,t)\,dv\,dt .
\end{aligned}
\end{equation}
Since
\[
\bigcup_{K=1}^{\infty}\operatorname{span}(\rrevised{\mathcal{H}_{i,K}})
\]
is dense in \(H^1(B_i)\), identity \eqref{eq:t94} implies that
\begin{equation}\label{eq:g99}
\partial_t\rho_i(\cdot,t)
=
\operatorname{div}\big(G_i(\cdot,t)\nabla\rho_i(\cdot,t)\big)
+
\left\langle J_i(\cdot,t),\nabla\rho_i(\cdot,t)\right\rangle
+
\rho_i(\cdot,t)
+
g_i(\cdot,t)
\end{equation}
in \(H^{-1}(B_i)\) for a.e. \(t\in[0,T]\).

}		
		Since $\rho_i\in W^{1,\infty}(0,T;\mathcal{C}^{m}(B_i))$, then by \cite[Section 5.9, Theorem 2]{evans2022partial}, we have $\rho_i\in \mathcal{C}([0,T];\mathcal{C}^{m}(B_i))$ and 
		\begin{equation}\label{eq:112}
			\rho_i(v,s')=\rho_i(v,s)+\int_{s}^{s'}\partial_t\rho_i(v,t)dt.
		\end{equation}
        {Note the above equality is understood in the classical sense,}
		so $\rho_i(v,t)$ is continuous in $t$. Moreover, for $l\leq m$, we have
		\begin{equation}
			\rho_i^{(l)}(v,s')=\rho_i^{(l)}(v,s)+\int_s^{s'}\partial_t\rho_i^{(l)}(v,t)dt.
		\end{equation}
		Therefore $\rho_i^{(l)}(v,t)$ is {Lipschitz continuous in $t$, for $l\leq m$, as $\partial_t\rho_i\in L^\infty(0,T;\mathcal{C}^m(B_i))$. Thus, the right hand sides of \eqref{eq:g99} is H\"older-$\alpha$-continuous in $t$,  since $\boldsymbol{J},G,\nabla G,g$ are locally H\"older-$\alpha$-continuous in time and $\rho_i^{(l)}$ is Lipschitz continuous in time. Hence, we have that the left hand sides of \eqref{eq:g99}, i.e., $\partial_t\rho_i(v,t)$, is H\"older-$\alpha$-continuous in $t$, that is} $\rho_i\in \mathcal{C}^{m,1+\alpha}(B_i\times [0,T])$, for any $m>0$ and $\rho_i$ is a classical solution to 
		\begin{equation}\label{eq:g102}
			\begin{aligned}	&\partial_t\rho_i=\div(G_i\nabla\rho_i)+\inner{\J_i}{\nabla\rho_i}+\rho_i+g_i,\\
				&\rho_i(v,0)=\varrho(v)(1-S(\norm{v}-n_i)).
			\end{aligned}
		\end{equation}

By taking $m$-th spatial derivative on both sides of \eqref{eq:g102}, we have
\begin{equation}\label{eq:sima104}
\begin{aligned}
		\partial_t\rho_i^{(m)}&=\sum_{s=0}^mC(s,m)\div(G_i^{(s)}\nabla\rho_i^{(m-s)})+\sum_{s=0}^mC(s,m)\inner{\J_i^{(s)}}{\nabla\rho_i^{(m-s)}}+\rho_i^{(m)}+g_i^{(m)},
\end{aligned}
\end{equation}
\rrevised{Thus, for any compact set \(Z\subset B_i\), we first take the squared norm of both sides of \eqref{eq:sima104}, and then integrate over \(Z\). Using the Cauchy--\rrevised{Schwarz} inequality, Assumption~\ref{asp:g1}, and the estimate \eqref{eq:g89}, we obtain}
		\begin{equation}\label{eq:g105}
		\begin{aligned}
				\int_{Z}\norm{\partial_t\rho_i^{(m)}}^2dv&\leq C(\norm{G}_{L^{\infty}([0,T]\times Z)},m)\Big(\norm{\rho_i}^2_{L^{\infty}(0,T;H^{\rrevised{m+2}}(Z))}+\norm{g}_{L^{\infty}(0,T;H^m(Z))}^2\Big)\\
			&\leq C(\norm{G}_{L^{\infty}([0,T]\times Z)},m)\Big(\norm{\rho_i}^2_{L^{\infty}(0,T;H^{m+2}(B_i))}+\norm{g}_{L^{\infty}(0,T;H^m(\mathbb{R}^d))}^2\Big)\\
			&\leq C(\norm{G}_{L^{\infty}([0,T]\times Z}),m,d,T,\norm{g}_{L^{\infty}(0,T;H^{m+2}(\mathbb{R}^d))},\norm{\varrho}_{L^{\infty}(0,T;H^{m+2}(\mathbb{R}^d))})\\
            &<\infty,
		\end{aligned}
		\end{equation}
		whose bound does not depend on $i$. {\rrevised{Specifically, in the last inequality we used the uniform bound from \eqref{eq:g89}.)} Here, we want to comment that the difference between \eqref{eq:g105} and \eqref{eq:gg106} lies that the dependency on $i$ in \eqref{eq:gg106} is through $n_i$, which is an upper bound for $\norm{G}_{L^{\infty}([0,T]\times B_i)}$; instead, here in \eqref{eq:g105}, this is replaced by $\norm{G}_{L^{\infty}([0,T]\times Z)}$, which does not depend on $i$ but it does depend on $Z$. This is readily understood, since in \eqref{eq:g105} we only take the integration over $Z$, while in \eqref{eq:gg106}, the integration is on $B_i$. }

{      
\begin{remark}\label{rmk:g47}
Finally, we want to remark that all the analysis in this section also hold for the following equation:
\begin{equation}\label{eq:ggg106}
\begin{aligned}
&\partial_t\rho_i=\div(G_i\nabla\rho_i)-\div({\J_i}{\rho_i})+\rho_i+g_i,\\
&\rho_i(v,0)=\varrho(v)(1-S(\norm{v}-n_i)),\forall v\in B_i.
\end{aligned}
\end{equation}
Similar to  \eqref{eq:99},\eqref{eq:sima83}, we have
\begin{equation}
\begin{aligned}
	\frac{1}{2}\frac{d}{dt}\intr\norm{\partial^m\rho_i}^2dv
	&=-\intr\inner{\nabla\rho_i^{(m)}}{\partial^m(G_i\nabla\rho_i)}dv-\intr\partial^m\Big(\div{(\J_i\rho_i)}\Big)\rho_i^{(m)}dv\\
	&\quad+\intr(\rho_i^{(m)})^2dv+\intr g^{(m)}_i\rho_i^{(m)} dv\\
&=-\intr\inner{\nabla\rho_i^{(m)}}{\partial^m(G_i\nabla\rho_i)}dv-\intr\partial^m\Big(\inner{\J_i}{\nabla\rho_i}\Big)\rho_i^{(m)}dv\\
&\quad+\intr(\rho_i^{(m)})^2dv+\intr g^{(m)}_i\rho_i^{(m)} dv-\intr\partial^m\Big(\div{(\J_i)}\rho_i\Big)\rho_i^{(m)}dv\\
&\leq -\intr G_i\norm{\nabla\rho_i^{(m)}}^2 dv-\intr\inner{\J_i}{\nabla\rho_i^{(m)}}\rho_i^{(m)}dv\\
&\quad+\intr(\rho_i^{(m)})^2dv+\intr g_i^{(m)}\rho_i^{(m)} dv\\
&\quad +C\norm{\intr\inner{\nabla\rho_i^{(m)}}{G_i^{(1)}\nabla\rho_i^{(m-1)}}dv}\\
&\quad+C\sum_{s=2}^m\norm{\intr \inner{\nabla\rho_i^{(m)}}{G_i^{(s)}\nabla\rho_i^{(m-s)}}dv}\\
&\quad +C\sum_{s=1}^m\norm{\intr\inner{\J_i^{(s)}}{\nabla\rho_i^{(m-s)}}\rho_i^{(m)}dv}\\
&\quad+C\sum_{s=1}^m\norm{\intr{\div(\J_i^{(s)})}{\rho_i^{(m-s)}}\rho_i^{(m)}dv}.
\end{aligned}
\end{equation}
The only difference from \eqref{eq:sima83} is the last term $\sum_{s=1}^m\norm{\intr{\div(\J_i^{(s)})}{\rho_i^{(m-s)}}\rho_i^{(m)}dv}$, which is less than $C\sum_{0\leq s\leq m}\intr \norm{\rho_i^{(s)}}^2dv+C\sum_{0\leq s< m}\intr G_i \norm{\rho_i^{(s)}}^2dv$. So \eqref{eq:g88} holds for this new equation and thus the esitmates \eqref{eq:g89} and \eqref{eq:g90}  also hold, for any $m\geq 0$. Similarly, the estimate for $\partial_t\rho_i^{(m)}$ and the rest derivation of this section hold.
\end{remark}
}
		\subsection{Limit solutions on $\mathbb{R}^d$ by letting $i\to\infty$}
		In the last section, {specifically in \eqref{eq:g105} and the comment after it}, we have proved, for any $m\geq 0$ and any compact domain $Z\subset B_i$, 
        \begin{equation}\footnotesize
            \norm{\rho_i}_{W^{1,\infty}(0,T; H^m(Z))}\leq C(\rrevised{\norm{G}_{L^{\infty}([0,T]\times Z)}},m,d,T,\norm{g}_{L^{\infty}(0,T;H^{m+2}(\mathbb{R}^d))},\norm{\varrho}_{L^{\infty}(0,T;H^{m+2}(\mathbb{R}^d))})<\infty
        \end{equation}
        which does not depend on $i$; Thus, by Sobolev embedding {and the derivation between \eqref{eq:112}-\eqref{eq:g102},} we have $\rho_i\in \mathcal{C}^{m,1+\alpha}(Z\times [0,T])$(its norm in $\mathcal{C}^{m,1+\alpha}(Z\times [0,T])$ is independent of $i$) and $\rho_i$ satisfies equalities \eqref{eq:g102} for any $v\in Z$. So by {a diagonal selection procedure}, when $i\to\infty$, we can find a limit function $\rho\in \mathcal{C}^{m',1+\alpha'}(Z\times [0,T])$ for $m'<m,\alpha'<\alpha$, that satisfies 
				\begin{equation}\label{eq:g106}
			\begin{aligned}	&\partial_t\rho=\div(G\nabla\rho)+\inner{\J}{\nabla\rho}+\rho+g,\\
				&\rho(v,0)=\varrho(v),
			\end{aligned}
		\end{equation}
{classically, for any $v\in \mathbb{R}^d$.} {By the estimates of \eqref{eq:g89}, \eqref{eq:g90}, and \eqref{eq:gg106},} it is not hard to verify that $\rho\in L^{\infty}(0,T;H^m(\mathbb{R}^d))$ for any $m\geq 0$,
\begin{equation}
\intt\int_{\mathbb{R}^d}G(v,t)\norm{\rho^{(m)}}^2dvdt<\infty,
\end{equation}
for any $m> 0$ and $\rho\in \mathcal{C}^{m,1+\alpha}(Z\times [0,T])\cap \rrevised{W^{1,\infty}(0,T;H^m(Z))}$, for any $m\geq 0$ and any compact region $Z\subset\mathbb{R}^d$. 
\begin{remark}\label{rmk:g95}
By Remark \ref{rmk:g47}, the above derivation and results also hold for equation \eqref{eqg:2}.
\end{remark}

\subsection{Improved estimates under stronger assumptions}\label{sec:304}
In the previous sections, under Assumptions \ref{asp:g1} and \ref{asp:g2}, we showed that equation \eqref{eq:g1}~(\eqref{eqg:2}) has a classical solution $\rho$ satisfying $\intt\intd G\norm{\rho^{(m)}}^2dvdt<\infty,\forall m>0,$ and $\rho\in W^{1,\infty}(0,T;H_{loc}^m(\mathbb{R}^d)),$\\$\forall m\geq 0$. In this section, we will show that this solution is also in function class $W^{1,\infty}(0,T;H^m(\mathbb{R}^d)),$  and satisfies $\intd G^2\norm{\rho^{(m)}}dv<\infty,\forall t\in [0,T],$$\intt\intd G^3\norm{\rho^{(m)}}^2dvdt<\infty$, for all $m\geq 0$. This result requires the following extra assumption.

\begin{assumption}\label{asp:g3}
   We assume 
  \begin{equation}
  	     \begin{aligned}
            &\intd (1+G^2(v,t))\norm{\varrho^{(m)}(v)}^2dv\leq C<\infty,\quad \forall m\geq 0, t\in [0,T],\\
  	         &\intd (1+G^2(v,t))\norm{g^{(m)}(v,t)}^2dv\leq C<\infty,\quad \forall m\geq 0, t\in [0,T],
  	     \end{aligned}
  \end{equation}
   and on each torus $B_i$, there exists a smooth non-negative function $Q_i:B_i\to \mathbb{R}$, which does not depend on $t$, such that 
   \begin{eqnarray}
       &\norm{\nabla Q_i}\leq C(1+Q_i^{\frac{1}{2}}),\\
        &C_1(1+G_i)\leq Q_i+1\leq C_2(1+G_i),
   \end{eqnarray}
   here constants do not depend on $i,t$.
\end{assumption}
\rrevised{We use the following weighted norms:
\[
\|g\|_{L^\infty(0,T;H^m(\mathbb{R}^d;1+G^2))}
:=
\sup_{t\in [0,T]}
\left\{
\sum_{s=0}^m
\int_{\mathbb{R}^d}
\|g^{(s)}(v,t)\|^2(1+G^2(v,t))\,dv
\right\},
\]
and
\[
\|\varrho\|_{L^\infty(0,T;H^m(\mathbb{R}^d;1+G^2))}
:=
\sup_{t\in [0,T]}
\left\{
\sum_{s=0}^m
\int_{\mathbb{R}^d}
\|\varrho^{(s)}(v)\|^2(1+G^2(v,t))\,dv
\right\}.
\]
Both quantities are finite under Assumption \ref{asp:g3}.
}
\begin{remark}\label{rmke:g211}
    A sufficient condition for the above assumption to be satisfied is that $g$ is a smooth compactly supported function on $\mathbb{R}^d\times (0,T)$, $\norm{\nabla G}\leq C\Big(1+G^{1/2}\Big)$ and $ G(v,t_2)+1\leq C(1+G(v,t_1))$ for any $v\in \mathbb{R}^d$ and any $t_1,t_2\in [0,T]$.{For example, in the CBO dynamics \eqref{eq:gcbo},\eqref{eq:CBOrev}, $G(v,t):=\norm{v-v_{\alpha}(t)}$ satisfies the above condition.} Then under the above sufficient condition,  we can choose $Q_i(\cdot):=G_i(\cdot,t)$, for any fixed $t\in [0,T]$; it can be proved that such a  $Q_i$ satisfies the above assumption. The verification is based on the analysis in Section \ref{sec:aspaddr}; to avoid tedious repetitions in the text, we have included it in Appendix \ref{apx:1}.
\end{remark}

Under Assumptions \ref{asp:g1}, \ref{asp:g2} and \ref{asp:g3}, we have the following results: On $B_i$, we can choose the test function $\psi(v,t):=\partial^{m}\Big({\Q_i}^2\rho^{(m)}_i\Big)$, then we have
	\begin{equation}\label{eq:sima120}
		\begin{aligned}
			\frac{1}{2}\frac{d}{dt}\intr {\Q_i}^2\norm{\rho_i^{(m)}}^2dv&=-\intr\inner{\partial^m(G_i\nabla\rho_i)}{\nabla(\Q_i^2\rho_i^{(m)})}dv\\
			&\quad+\intr\partial^{m}\Big(\inner{\J_i}{\nabla\rho_i}\Big)\Q_i^2\rho_i^{(m)}dv\\
			&\quad+\intr {\Q_i}^2\norm{\rho_i^{(m)}}^2dv+\intr Q_i^2\partial^m g_i\rho_i^{(m)}dv,
		\end{aligned}
	\end{equation}
\paragraph{Estimate the First Term on the Right Hand Sides of \eqref{eq:sima120}.}
For the first term on the right hand side, we have
\begin{equation}
	\begin{aligned}
	-\intr\inner{\partial^m(G_i\nabla\rho_i)}{\nabla(\Q_i^2\rho_i^{(m)})}dv&\leq -\intr G_i\Q_i^2\norm{\nabla\rho_i^{(m)}}^2dv\\
	&\quad +C\sum_{s=1}^m\norm{\intr\inner{G_i^{(s)}\nabla\rho_i^{(m-s)}}{\Q_i^2\nabla\rho_i^{(m)}}dv}\\
	&\quad +C\norm{\intr\inner{G_i\nabla\rho_i^{(m)}}{\Q_i\rho_i^{(m)}\nabla \Q_i}dv}\\
	&\quad+C\sum_{s=1}^m\norm{\intr\inner{G_i^{(s)}\nabla\rho_i^{(m-s)}}{\Q_i\rho_i^{(m)}\nabla\Q_i}}\\
	&=I+II+III+IV.
	\end{aligned}
\end{equation}	When $m=0$, $II=0,IV=0$, so in the next, when estimating $II,IV$, we will assume $m\geq 1$.

\paragraph{Term $II$.} use Assumption \ref{asp:g1}, \ref{asp:g3} and Cauchy-\rrevised{Schwarz} inequality to obtain 
	\begin{equation}
		\begin{aligned}
			II&\leq C\sum_{i=0}^{m-1}\intr (1+G_i)Q_i^2\norm{\nabla\rho_i^{(s)}}\norm{\nabla\rho_i^{(m)}}dv\\
            &\leq\epsilon\intr G_i\Q_i^2\norm{\nabla\rho_i^{(m)}}^2dv+C(\epsilon)\sum_{s=0}^{m-1}\intr G_i\Q_i^2\norm{\nabla\rho_i^{(s)}}^2dv\\
            &\quad+C\sum_{s=0}^{m-1}\intr Q_i^3\norm{\nabla \rho_i^{(s)}}^2dv+C\intr Q_i\norm{\nabla\rho^{(m)}}^2dv\\
            &\leq \epsilon\intr G_i\Q_i^2\norm{\nabla\rho_i^{(m)}}^2dv+C(\epsilon)\sum_{s=0}^{m-1}\intr G_i\Q_i^2\norm{\nabla\rho_i^{(s)}}^2dv+C\sum_{s=0}^m\intr \norm{\nabla \rho_i^{(s)}}^2dv\\
            &\quad+C\intr G_i\norm{\nabla\rho_i^{(m)}}^2dv,
		\end{aligned}
	\end{equation}
    where we used $Q_i^3\leq C(1+G_i)Q_i^2\leq CG_iQ_i^2+C$, since $CQ_i^2\leq CG_iQ_i^2+C$: if $G_i\geq 1$, we have $CQ_i^2\leq CG_iQ_i^2$, else $CQ_i^2\leq C(G_i+1)^2\leq C$; 
    
    \paragraph{Term $IV$.} use Assumption \ref{asp:g1}, \ref{asp:g3} and Cauchy-\rrevised{Schwarz} inequality, we have
	\begin{equation}
		\begin{aligned}
			IV&\leq C\sum_{s=0}^{m-1}\intr (1+G_i)Q_i(1+Q_i^{\frac{1}{2}})\norm{\nabla \rho_i^{(s)}}\norm{\rho_i^{(m)}}dv\\
            &\leq \underbrace{C\sum_{s=0}^{m-1}\intr G_iQ_i(1+Q_i^{\frac{1}{2}})\norm{\nabla \rho_i^{(s)}}\norm{\rho_i^{(m)}}dv}_{A}\\
            &\quad+\underbrace{C\sum_{s=0}^{m-1}\intr (1+G_i^{\frac{1}{2}})\norm{\nabla \rho_i^{(s)}}\norm{\rho_i^{(m)}}dv}_{B},
		\end{aligned}
	\end{equation}
    since $(1+G_i)Q_i(1+Q_i^{\frac{1}{2}})\leq CG_iQ_i(1+Q_i^{\frac{1}{2}})+C(1+G_i^{\frac{1}{2}})$, this is due to $(1+G_i)Q_i\leq CG_iQ_i+C$ and $1+Q_i^{\frac{1}{2}}\leq C(1+G_i^{\frac{1}{2}})$;
    for term $A$, we have
    \begin{equation}
        \begin{aligned}
             A&\leq C\sum_{s=0}^{m-1}\intr G_iQ_i^2\norm{\nabla\rho_i^{(s)}}^2dv+C\intr G_i\norm{\rho_i^{(m)}}^2dv \\
            &\quad+C\sum_{s=0}^{m-1}\intr G_iQ_i^2\norm{\nabla\rho_i^{(s)}}^2dv+C\intr G_iQ_i\norm{\rho_i^{(m)}}^2dv\\
            &\leq C\sum_{s=0}^{m-1}\intr G_iQ_i^2\norm{\nabla\rho_i^{(s)}}^2dv+C\intr Q_i^2\norm{\rho_i^{(m)}}^2dv+C\intr G_i\norm{\rho_i^{(m)}}^2dv+C\intr \norm{\rho_i^{(m)}}^2dv,
        \end{aligned}
    \end{equation}
    since $G_iQ_i\leq C(1+Q_i)Q_i\leq CQ_i^2+CG_i+C$; for term $B$, use Cauchy-\rrevised{Schwarz} inequality, we have
    \begin{equation}
        \begin{aligned}
            B\leq C\sum_{s=0}^{m-1}\intr\norm{\nabla\rho_i^{(s)}}^2dv+C\intr G_i\norm{\rho_i^{(m)}}^2dv;
        \end{aligned}
    \end{equation}
    so for term $IV$, we get the estimate
    \begin{equation}
        \begin{aligned}
            IV&\leq C\sum_{s=0}^{m-1}\intr G_iQ_i^2\norm{\nabla\rho_i^{(s)}}^2dv+C\intr Q_i^2\norm{\rho_i^{(m)}}^2dv+C\intr G_i\norm{\rho_i^{(m)}}^2dv\\
            &\quad+C\sum_{s=0}^{m-1}\intr \norm{\nabla \rho_i^{(s)}}^2dv;
        \end{aligned}
    \end{equation}
\paragraph{Term $III$.} 
	use Assumption \ref{asp:g3} and Cauchy-\rrevised{Schwarz} inequality, we have
	\begin{equation}
		\begin{aligned}
			III&\leq C\intr G_iQ_i(1+Q_i^{\frac{1}{2}})\norm{\nabla\rho_i^{(m)}}\norm{\rho_i^{(m)}}dv\\
            &\leq C\intr G_iQ_i\norm{\nabla\rho_i^{(m)}}\norm{\rho_i^{(m)}}dv+C\intr G_iQ_iQ_i^{\frac{1}{2}}\norm{\nabla\rho_i^{(m)}}\norm{\rho_i^{(m)}}dv\\
            &\leq \frac{1}{2}\epsilon\intr G_i\Q_i^2\norm{\nabla\rho_i^{(m)}}^2dv+C(\epsilon)\intr G_i\norm{\rho_i^{(m)}}^2dv\\
            &\quad+\frac{1}{2}\epsilon\intr G_i\Q_i^2\norm{\nabla\rho_i^{(m)}}^2dv+C(\epsilon)\intr G_iQ_i\norm{\rho_i^{(m)}}^2dv \\
            &\leq \epsilon\intr G_i\Q_i^2\norm{\nabla\rho_i^{(m)}}^2dv+C(\epsilon)\Big[\intr Q_i^2\norm{\rho_i^{(m)}}^2dv+\intr G_i\norm{\rho_i^{(m)}}^2dv+\intr \norm{\rho_i^{(m)}}^2dv\Big],
		\end{aligned}
	\end{equation}
    where again, we used $G_iQ_i\leq CQ_i^2+CG_i+C$.

    All in all, for the first term in the right hand side of \eqref{eq:sima120}, when $m=0$,  we have it less than
	\begin{equation}
		\begin{aligned}
			-(1-\epsilon)\intr G_i\Q_i^2\norm{\nabla\rho_i}^2dv+C(\epsilon)\Big[\intr \Q_i^2(\rho_i)^2dv+\intr G_i(\rho_i)^2dv+\intr(\rho_i)^2dv\Big];
		\end{aligned}
	\end{equation}
    while $m\geq 1$, we have it upper bounded by
	\begin{equation}
		\begin{aligned}
			&-(1-2\epsilon)\intr G_i\Q_i^2\norm{\nabla\rho_i^{(m)}}^2dv\\
			&+C(\epsilon)\Big[\sum_{s=0}^{m-1}\intr G_i\Q_i^2\norm{\nabla\rho_i^{(s)}}^2dv+\intr Q_i^2\norm{\rho_i^{(m)}}^2dv+\intr G_i\norm{\rho_i^{(m)}}^2dv+\intr \norm{\rho_i^{(m)}}^2dv\Big]\\
            &+C\intr G_i\norm{\nabla\rho_i^{(m)}}^2dv+C\sum_{s=0}^m\intr\norm{\nabla\rho_i^{(s)}}^2dv.
		\end{aligned}
	\end{equation}

\paragraph{Estimate the Second Term on the Right Hand Sides of \eqref{eq:sima120}.}
 when $m=0$, we have
	\begin{equation}\label{eq:sima130}
			\begin{aligned}
				\intr\partial^{m}\Big(\inner{\J_i}{\nabla\rho_i}\Big)\Q_i^2\rho_i^{(m)}dv
				&\leq \norm{\intr \inner{\J_i}{\nabla\rho_i}\Q_i^2\rho_idv}\\
				&\leq \epsilon\intr G_i\Q_i^2\norm{\nabla\rho_i}^2dv+C(\epsilon)\intr \Q_i^2(\rho_i)^2dv;
			\end{aligned}
	\end{equation}
	while, for $m\geq 1$, we use Assumption \ref{asp:g1},\ref{asp:g3} and Cauchy-\rrevised{Schwarz} inequality to obtain
	\begin{equation}\label{eq:sima131}
		\begin{aligned}
			&\intr\partial^{m}\Big(\inner{\J_i}{\nabla\rho_i}\Big)\Q_i^2\rho_i^{(m)}dv\\&\leq C\norm{\intr \inner{\J_i}{\nabla\rho_i^{(m)}}\Q_i^2\rho_i^{(m)}dv}+C\sum_{s=1}^m\norm{\intr \inner{\J_i^{(s)}}{\nabla\rho_i^{(m-s)}}\Q_i^2\rho_i^{(m)}dv}\\
            &\leq C\intr {G_i^{\frac{1}{2}}}\norm{\nabla\rho_i^{(m)}}\Q_i^2\norm{\rho_i^{(m)}}dv+C\sum_{s=1}^m\intr (1+G_i^{\frac{1}{2}})\norm{\nabla\rho_i^{(m-s)}}\Q_i^2\norm{\rho_i^{(m)}}dv\\
			&\leq \epsilon\intr G_i\Q_i^2\norm{\nabla\rho_i^{(m)}}^2dv+C(\epsilon)\intr \Q_i^2\norm{\rho_i^{(m)}}^2dv\\
			&\quad+C\Big[\sum_{s=0}^{m-1}\intr G_i\Q_i^2\norm{\nabla\rho_i^{(s)}}^2dv+\sum_{s=0}^{m-1}\intr Q_i^2\norm{\nabla\rho_i^{(s)}}^2dv+\intr \Q_i^2\norm{\rho_i^{(m)}}^2dv\Big].
		\end{aligned}
	\end{equation}
\paragraph{Estimate the Last Term on the Right Hand Sides of \eqref{eq:sima120}.} use Assumption \ref{asp:g3} and Cauchy-\rrevised{Schwarz} inequality, the last term on the right hand side can be estimated by
    \begin{equation}
       \begin{aligned}
            \intr Q_i^2\norm{\rho_i^{(m)}}^2dv+\intr Q_i^2\norm{g_i^{(m)}}^2dv&\leq \intr Q_i^2\norm{\rho_i^{(m)}}^2dv+C\sum_{s=0}^m\intr Q_i^2\norm{g^{(s)}}^2dv\\
            &\leq \intr Q_i^2\norm{\rho_i^{(m)}}^2dv+C\sum_{s=0}^m\intd(1+G^2)\norm{g^{(s)}}^2dv\\
            &\leq \intr Q_i^2\norm{\rho_i^{(m)}}^2dv+C.
       \end{aligned}
    \end{equation}
 \paragraph{Combine the estimates of the Right Hand Sides of \eqref{eq:sima120}.}    
	combining the above estimates, when $m=0$, we have
	\begin{equation}
		\begin{aligned}
		    \frac{d}{dt}\intr \Q_i^2(\rho_i)^2dv&\leq -2(1-2\epsilon)\intr G_i\Q_i^2\norm{\nabla\rho_i}^2dv\\
            &\quad+C(\epsilon)\Big[\intr \Q_i^2(\rho_i)^2dv+\intr G_i(\rho_i)^2dv+\intr(\rho_i)^2dv\Big]+C,
		\end{aligned}
	\end{equation}
	then choosing $\epsilon$ small enough, combining the estimates of \eqref{eq:g89}, \eqref{eq:g90} and using Gr\"onwall lemma, we have
\rrevised{\begin{eqnarray}
	&	\intr \Q_i^2(\rho_i)^2\leq C(T,\norm{g}_{L^{\infty}(0,T;L^2(\mathbb{R}^d;1+G^2))},\norm{\varrho}_{L^\infty(0,T;L^2(\mathbb{R}^d;1+G^2))})<\infty,\label{eq:sima95}\\
	&\intt\intr G_i\Q_i^2\norm{\nabla\rho_i}^2dvdt\leq C(T,\norm{g}_{L^{\infty}(0,T;L^2(\mathbb{R}^d;1+G^2))},\norm{\varrho}_{L^\infty(0,T;L^2(\mathbb{R}^d;1+G^2))})<\infty,\label{eq:sima96}
\end{eqnarray}}
	where constant $C$ does not depend on $i$; when $m\geq 1$, we have
	\begin{equation}\label{eq:g117}
		\begin{aligned}
			&\sum_{s=m}\intr\Q_i^2(v)\norm{\rho_i^{(s)}(v,t)}^2dv-\sum_{s=m}\intr\Q_i^2(v)\norm{\rho_i^{(s)}(v,0)}^2dv\\
            &\leq -2(1-3\epsilon)\sum_{s=m}\int_0^t\intr G_i(v,r)Q_i^2(v)\norm{\nabla\rho_i^{(s)}(v,r)}^2dvdr\\
			&\quad+C(\epsilon)\sum_{s=m}\intr\Q_i^2(v)\norm{\rho_i^{(s)}(v,r)}^2dvdr+C(\epsilon,m),
		\end{aligned}
	\end{equation}
    here $C(\epsilon,m)$ depends on $\epsilon$ and depends linearly on
    \begin{eqnarray*}
        &\intt\intr Q_i^2\norm{\rho^{(s)}_i}^2dv,\\
        &\intt\intr G_i\Q_i^2\norm{\nabla\rho^{(s)}_i}^2dvdt,\\
        &\intt\intr \norm{\rho_i^{(s)}}^2dv,\\
        &\intt\intr G_i\norm{\nabla\rho^{(s)}_i}^2dvdt,\\
        &\intt\intd (1+G^2)\norm{g^{(s)}}^2dvdt,
    \end{eqnarray*}
    for any $0\leq s<m$.
	So, by induction, choosing $\epsilon$ small enough, {combining the estimates of \eqref{eq:g89}, \eqref{eq:g90},\eqref{eq:sima95},\eqref{eq:sima96}} and using  Gr\"onwall lemma, we have 
\rrevised{
{\footnotesize
\begin{align}
    \int_{\mathbb{R}^d} Q_i^2 \norm{\rho_i^{(m)}}^2\,dv
    &\leq 
    C\Big(
        T,m,d,
        \norm{g}_{L^{\infty}(0,T;H^m(\mathbb{R}^d;1+G^2))},
        \norm{\varrho}_{L^\infty(0,T;H^m(\mathbb{R}^d;1+G^2))}
    \Big)<\infty,
    \label{eq:g118}\\
    \int_0^T\int_{\mathbb{R}^d}
    G_i Q_i^2 \norm{\nabla\rho_i^{(m)}}^2\,dvdt
    &\leq 
    C\Big(
        T,m,d,
        \norm{g}_{L^{\infty}(0,T;H^m(\mathbb{R}^d;1+G^2))},
        \norm{\varrho}_{L^\infty(0,T;H^m(\mathbb{R}^d;1+G^2))}
    \Big)<\infty .
    \label{eq:g119}
\end{align}
}
}
	for any $m\geq 0$ and constants do not depend on $i$.

    Now, with the above estimate, we give a new upper for $\intr\norm{\partial_t\rho_i^{(m)}}^2dv$, which is independent of $i$. We have
		\begin{equation}
		\begin{aligned}
			\intr\norm{\partial_t\rho_i^{(m)}}^2dv
			&=\intr \partial^m\inner{\J_i}{\nabla\rho_i}\partial_t\rho_i^{(m)} dv+\intr\partial^m\div{(G_i\rho_i)}\partial_t\rho^{(m)}dv\\
			&\leq \epsilon\intr\norm{\partial_t\rho_i^{(m)}}^2dv+\frac{C}{\epsilon}\Big(\norm{\rho_i}^2_{L^{\infty}(0,T;H^{m+2}(B_i))}+\sum_{s\leq m+2}\intr\Q_i^2\norm{\rho_i^{(s)}}^2dv\Big),
		\end{aligned}
	\end{equation}
	and therefore, using the estimates \eqref{eq:g89} and \eqref{eq:g118}, and choosing $\epsilon>0$ sufficiently small, we obtain
\rrevised{
\begin{equation}\label{eq:k102}
\begin{aligned}
    \int_{\mathbb{R}^d}\norm{\partial_t\rho_i^{(m)}}^2\,dv
    &\leq
    \frac{C}{\epsilon(1-\epsilon)}
    \left(
        \norm{\rho_i}^2_{L^{\infty}(0,T;H^{m+2}(B_i))}
        +
        \sum_{s\leq m+2}
        \int_{\mathbb{R}^d}
        Q_i^2\norm{\rho_i^{(s)}}^2\,dv
    \right) \\
    &\leq
    C\Big(
        T,m+2,d,
        \norm{g}_{L^{\infty}(0,T;H^{m+2}(\mathbb{R}^d;1+G^2))},
        \norm{\varrho}_{L^\infty(0,T;H^{m+2}(\mathbb{R}^d;1+G^2))}
    \Big) \\
    &<\infty .
\end{aligned}
\end{equation}

In the second inequality of \eqref{eq:k102}, we used the estimate \eqref{eq:g89}.} In particular, the constant $C$ is independent of $i$. Let $i\to\infty$ and conclude that there is a solution $\rho$ to equation \eqref{eq:g1}, which is in $W^{1,\infty}(0,T;H^m(\mathbb{R}^d))$ for any $m\geq 0$ and also in function class $\Gamma_G$.
    \begin{remark}\label{rmk:g2211}
        For equation \eqref{eqg:2}, we should estimate $\div{(\J\rho)}$ in the following way:
        When $m=0$,  use integration by parts, we have
        \begin{equation}
        	-\intr\div{(J_i\rho_i)}Q_i^2\rho_i dv={\intr \inner{J_i}{\nabla\rho_i}Q_i^2\rho_i dv}+2{\intr \inner{J_i}{\nabla Q_i}Q_i(\rho_i)^2 dv};
        \end{equation}
       we can use \eqref{eq:sima130} to bound the first term, that is 
       \begin{eqnarray}
       	D\leq \epsilon\intr G_i\Q_i^2\norm{\nabla\rho_i}^2dv+C(\epsilon)\intr \Q_i^2(\rho_i)^2dv;
       \end{eqnarray}
       for the second term, use Assumption \ref{asp:g1}, \ref{asp:g3} and Cauchy-\rrevised{Schwarz} inequality, we have
       \begin{equation}
       	\begin{aligned}
       		E&\leq C\intr (\rho_i)^2dv+\intr Q_i^2(\rho_i)^2dv.
       	\end{aligned}
       \end{equation}
       When $m\geq 1$, we have
        \begin{equation}
            -\intr\partial^m\div(\J_i\rho_i)Q_i^2\rho_i^{(m)}dv={-\intr\partial^m\Big(\inner{\J_i}{\nabla\rho_i}\Big)Q_i^2\rho_i^{(m)}dv-\intr\partial^m\Big(\div(\J_i)\rho_i\Big)Q_i^2\rho_i^{(m)}dv},
        \end{equation}
        for the first term, we have already given a upper bound, see \eqref{eq:sima131}; we can bound the  second term  by
        \begin{equation}
            \begin{aligned}
                -\intr\partial^m\Big(\div(\J_i)\rho_i\Big)Q_i^2\rho_i^{(m)}dv&\leq C\sum_{s=0}^m\norm{\intr \partial^{s}(\div(\J_i))\rho_i^{(m-s)}Q_i^2\rho_i^{(m)}dv}\\
                &\leq C\sum_{s=0}^m\norm{\intr (1+G_i^{\frac{1}{2}})\rho_i^{(m-s)}Q_i^2\rho_i^{(m)}dv}\\
                &\leq C\sum_{s=0}^m\norm{\intr G_i^{\frac{1}{2}}\rho_i^{(s)}Q_i^2\rho_i^{(m)}dv}+ C\sum_{s=0}^m\norm{\intr \rho_i^{(s)}Q_i^2\rho_i^{(m)}dv}\\
                &\leq C\sum_{s=0}^m\Big[\intr G_iQ_i^2\norm{\rho_i^{(s)}}^2dv+\intr Q_i^2\norm{\rho_i^{(s)}}^2dv\Big].
            \end{aligned}
        \end{equation}
        So all in all, the estimate of the second term of equation \eqref{eqg:2}  can also be absorbed into the right hand side of \eqref{eq:g117}, thus we also have the estimates \eqref{eq:g118}, \eqref{eq:g119}.
    \end{remark}

\subsection{Uniqueness of weak solutions}\label{sec:305}
In this section, we prove the uniqueness of the weak solution under Assumptions \ref{asp:g1},\ref{asp:g2} and \ref{asp:g3}.

We first prove a lemma which guarantees that we can use integration by parts on $\mathbb{R}^d$.

\begin{lemma}\label{lm:27}Under Assumption \ref{asp:g1}, let $\phi,\psi\in\Gamma_G'$, here 
	\begin{equation}
		\begin{aligned}
			\Gamma_G'&:=\Big\{\phi, \phi\in \mathcal{C}^{2,1}(\mathbb{R}^d\times [0,T])\cap L^{\infty}(0,T;H^1(\mathbb{R}^d)): \intt\intd G\norm{\phi}^2dvdt<\infty,\\
			&\quad \intt\intd G\norm{\nabla\phi}^2dvdt<\infty, \intt\intd G\sum_{i=1}^d\sum_{j=1}^d|\partial_i\partial_j\phi|^2dvdt<\infty\Big\}.
		\end{aligned}
	\end{equation}
	Then we have
	\begin{align}
		&\intt\intd \partial_t\phi\psi dvdt=\intd \phi(v,T)\psi(v,T)dv-\intd \phi(v,0)\psi(v,0)dv-\intt\intd \partial_t\psi\phi dvdt;\\
		&\intt\intd\div\left(\J\phi\right)\psi dvdt=-\intt\intd \inner{\J}{\nabla\psi}\phi dvdt; \label{2nd}\\
		&\intt\intd \Delta\left(G\phi\right)\psi dvdt=-\intt\intd \inner{\nabla\left(G\phi\right)}{\nabla\psi}dvdt; \label{3rd}\\
		&\intt\intd G\phi\Delta\psi dvdt=-\intt\intd \inner{\nabla\left(G\phi\right)}{\nabla\psi}dvdt.\label{4th}
	\end{align}
	
\end{lemma}
\begin{proof}
	The first identity is natural. We only focus on the rest of the equalities. For the second one \eqref{2nd}, we use finite cubes to support approximations. For any $R>0$, let $Q_{R}(0):=[-R,R]^d$, then by the divergence theorem, we have
	\begin{equation*}
		\intt\int_{Q_R(0)} \div\left(\J\phi\right)\psi dvdt=\intt\int_{\partial Q_R(0)}\inner{\J}{n}\phi\psi dSdt-\intt\int_{Q_R(0)} \inner{\J}{\nabla\psi}\phi dvdt,
	\end{equation*}
	here $n$ is the unit outer normal vector of $\partial Q_R(0)$. Next, we will show
	\begin{equation*}
		\lim_{R\to\infty}\Big|\intt\int_{\partial B_R(0)}\inner{\J}{n}\phi\psi dSdt\Big|=0,
	\end{equation*}
	then the second identity is proved by letting $R\to\infty$; we have
	\begin{equation*}
		\begin{aligned}
			&\Big|\intt\int_{\partial Q_R(0)}\inner{\J}{n}\phi\psi dSdt\Big|\\
			&\leq \intt\int_{\partial Q_R(0)}\norm{\J}^2\phi^2dSdt+\intt\int_{\partial Q_R(0)}\psi^2dSdt\\
			&\leq \intt\int_{\partial Q_R(0)\cup\partial Q_{R+1}(0)}\norm{\J}^2\phi^2dSdt+\intt\int_{\partial Q_R(0)\cup\partial Q_{R+1}(0)}\psi^2dSdt\\
			&\leq C\Big(\intt\int_{Q_{R+1}(0)\setminus Q_R(0)}\norm{\J}^2\phi^2+\norm{\nabla(\J\phi)}^2dvdt+\intt\int_{Q_{R+1}(0)\setminus Q_R(0)}\psi^2+\norm{\nabla\psi}^2dvdt\Big)\\
			&\leq C\Big(\intt\int_{Q_{R+1}(0)\setminus Q_R(0)}\phi^2+G\phi^2+G\norm{\nabla\phi}^2dvdt+\intt\int_{Q_{R+1}(0)\setminus Q_R(0)}\psi^2+\norm{\nabla\psi}^2dvdt\Big)\\
			&\to 0,\quad \text{as } R\to\infty,
		\end{aligned}
	\end{equation*}
	since $\phi,\psi\in\Gamma'_G$. In the above we used the trace theorem and the constant $C$ in the third inequality is from the trace theorem; we want to comment that $C$ depends on $Q_{1/2}(0)$ and does not depend on $R$ here. Indeed, we can decompose $Q_{R+1}(0)\setminus Q_R(0)$ into disjoint cubes with side length $1$, then on each of these unit cubes, we use the trace theorem then take summation to get the third inequality above; the fourth inequality  above follows by Assumption \ref{asp:g1}. 
	
	The proof of the third equality \eqref{3rd}  is similar; we have for any finite cube that
	\begin{equation*}
		\begin{aligned}
			&\intt\int_{Q_R(0)} \Delta\left(G\phi\right)\psi dvdt=\intt\int_{\partial Q_R(0)}G\inner{\nabla\phi}{n}\psi dSdt\\
			&\quad+\intt\int_{\partial Q_R(0)}\frac{\partial G}{\partial n}\phi\psi dSdt-\intt\int_{Q_R(0)} \inner{\nabla\left(G\phi\right)}{\nabla\psi}dvdt;
		\end{aligned}
	\end{equation*}
	we show now that the second term on the right hand side, that is $\intt\int_{\partial Q_R(0)}{\partial G}/{\partial n}\phi\psi dSdt$, converges to $0$ as $R\to\infty$; using $\norm{\frac{\partial G}{\partial n}}\leq CG^{1/2}(1+G^{1/2})$ and the Cauchy-\rrevised{Schwarz} inequality, we have
	\begin{equation}
		\begin{aligned}
			&\intt\int_{\partial Q_R(0)}\frac{\partial G}{\partial n}\phi\psi dSdt\\
			&\leq C\Big(\intt \int_{\partial Q_R(0)}G\phi^2dvdt+\intt \int_{\partial Q_R(0)}\psi^2dvdt+\intt \int_{\partial Q_R(0)}G\psi^2dvdt\Big)\\
			&\leq C\Big(\intt\int_{Q_{R+1}(0)\setminus Q_R(0)}G\phi^2+\norm{\nabla(G^{\frac{1}{2}}\phi)}^2dvdt+\intt\int_{Q_{R+1}(0)\setminus Q_R(0)}G\psi^2+\norm{\nabla(G^{\frac{1}{2}}\psi)}^2dvdt\Big)\\
			&\quad+C\Big(\intt\int_{Q_{R+1}(0)\setminus Q_R(0)} \psi^2+\norm{\nabla\psi}^2dvdt\Big)\\
			&\leq C\Big(\intt\int_{Q_{R+1}(0)\setminus Q_R(0)}\phi^2+\norm{\nabla\phi}^2+\psi^2+\norm{\nabla\psi}^2dvdt\Big)\\
			&\quad+C\Big(\intt\int_{Q_{R+1}(0)\setminus Q_R(0)}G\phi^2+G\norm{\nabla\phi}^2+G\psi^2+G\norm{\nabla\psi}^2dvdt\Big)\\
			&\to 0,\quad \text{as }R\to\infty.
		\end{aligned}
	\end{equation}
	Above we used the inequality $\norm{\nabla G}\leq CG^{1/2}(1+G^{1/2})$ from Assumption \ref{asp:g1}, which is equivalent to $\norm{\nabla G^{1/2}}\leq C(1+G^{1/2})$;  the first term, that is $\intt\int_{\partial Q_R(0)}G\inner{\nabla\phi}{n}\psi dSdt$, converges to $0$ as $R\to\infty$:  just denote $\phi':=\inner{\nabla\phi}{n}$, then $\intt\intd G\norm{\nabla\phi'}^2dvdt\leq C\sum_{i,j}\intt\intd G\norm{\partial_i\partial_j\phi}^2dvdt<\infty$ and the rest proof is similar as  above; so we proved the third equality.
	
	For the last equality \eqref{4th}, it is enough to show 
	\begin{equation}
		\intt\int_{\partial Q_R(0)}G\phi\frac{\partial\psi}{\partial n}dvdt\to 0\quad \text{as } R\to\infty,
	\end{equation}
	which is true by using a similar argument as for the proof of the third equality.
\end{proof}
\begin{remark}
	Obviously, $\Gamma_G\subset\Gamma_G'$, so the above lemma holds for functions in $\Gamma_G$.
\end{remark}

\rrevised{We denote by $\tilde{\rho}$ a classical} solution of equation \eqref{eq:g1} in the function class $\Gamma_G$, which is guaranteed to exist by previous sections. By Lemma \ref{lm:27}, $\tilde{\rho}$ is also a weak solution. Let $\rho$ be any weak solution to equation \eqref{eq:g1}. We will show $\rho=\tilde{\rho}$ and thus the uniqueness is proved.

We assume $\psi$ be the solution of the following equation
\begin{equation}\label{eq:gg119}
    \begin{aligned}
&\partial_t\psi(v,t)+\div(G(v,t)\nabla\psi(v,t))-\div({\J(v,t)}{\psi(v,t)})+\psi(v,t)=h(v,t)\\
        &{\psi}(v,T)=0,
    \end{aligned}
\end{equation}
	where $h(v,t):=\eta_{R,\delta}(z^{\epsilon})^{\epsilon}$,
    $\eta_{R,\delta}(v,t):=\Big(1-S(\norm{v}-R)\Big)\cdot\chi_{[\delta,T-\delta]}^{\delta^2}(t), z(v,t):=\eta_{R,\delta}(v,t)\Big(\rho(v,t)-\tilde{\rho}(v,t)\Big)$ and $p^\eta(v,t):=(p*\varphi_\eta)(v,t)$, thus $h(v,t)$ has compact support on $\mathbb{R}^d\times (0,T)$.
    
    If we let $\tilde{\psi}(v,t):=\psi(v,T-t)$, then $\tilde{\psi}$ satisfies the following equation
		\begin{equation}
		\begin{aligned}
			&\partial_t{\tilde{\psi}}=\div\Big(\tilde{G}\nabla\tilde{\psi}\Big)-\div\Big(\tilde{\J}\tilde{\psi}\Big)+\tilde{\psi}+\tilde{h}\\
			&\tilde{\psi}(v,0)=0,
		\end{aligned}
	\end{equation}
	here $\tilde{\J}(v,t):=\J(v,T-t),\tilde{G}(v,t):=G(v,T-t),\tilde{h}(v,t):=-h(v,T-t)$. Hence, by Remark \ref{rmk:g2211}, the above equation has a classical solution in function class $\Gamma_G$, and so \eqref{eq:gg119} has a classical solution $\psi\in\Gamma_G$. Since ${\psi}\in \Gamma_G$ and $\rho,\tilde{\rho}$ are both weak solutions,  so by Definition \ref{def:sima24}, we have
	    \begin{equation}\label{eq:gg99}
        \begin{aligned}
            &\intt\intd\Big(\partial_t\psi(v,t)+\div(G(v,t)\nabla\psi(v,t))-\div({\J(v,t)}{\psi(v,t)})+\psi(v,t)\Big)\rho(v,t)dvdt\\
            &=\intt\intd h(v,t)\rho(v,t)dvdt=-\intd\varrho(v)\psi(v,0)dv-\intt\intd g(v,t)\psi(v,t)dvdt.
        \end{aligned}
    \end{equation}
	and
		    \begin{equation}\label{eq:gg100}
        \begin{aligned}
            &\intt\intd\Big(\partial_t\psi(v,t)+\div(G(v,t)\nabla\psi(v,t))-\div({\J(v,t)}{\psi(v,t)})+\psi(v,t)\Big)\tilde{\rho}(v,t)dvdt\\
            &=\intt\intd h(v,t)\tilde{\rho}(v,t)dvdt=-\intd\varrho(v)\psi(v,0)dv-\intt\intd g(v,t)\psi(v,t)dvdt.
        \end{aligned}
    \end{equation}
\rrevised{Subtracting \eqref{eq:gg100} from \eqref{eq:gg99} and using \eqref{eq:gg119}, we obtain
\begin{equation}\label{eq:gg101}
    \intt\intd h(v,t)\big(\rho(v,t)-\tilde{\rho}(v,t)\big)\,dvdt=0.
\end{equation}
By the definition of $h$ and $z$, this gives
\begin{equation}\label{eq:gg102}
\begin{aligned}
    0
    &=\intt\intd \eta_{R,\delta}(v,t)
    \big(\rho(v,t)-\tilde{\rho}(v,t)\big)
    (z^\epsilon)^\epsilon(v,t)\,dvdt \\
    &=\intt\intd z(v,t)(z^\epsilon)^\epsilon(v,t)\,dvdt .
\end{aligned}
\end{equation}
 Since the mollifier is chosen to be symmetric, convolution is self-adjoint in $L^2$. Therefore
\begin{equation}\label{eq:gg103}
    \intt\intd z(v,t)(z^\epsilon)^\epsilon(v,t)\,dvdt
    =
    \intt\intd \big(z^\epsilon(v,t)\big)^2\,dvdt .
\end{equation}
Combining \eqref{eq:gg102} and \eqref{eq:gg103}, we obtain
\begin{equation}\label{eq:gg104}
    \intt\intd \big(z^\epsilon(v,t)\big)^2\,dvdt=0.
\end{equation}
Hence $z^\epsilon=0$ almost everywhere in $\mathbb{R}^d\times(0,T)$.

We now let $\epsilon\to0$. For fixed $R>0$ and $\delta>0$, the function
\[
    z=\eta_{R,\delta}(\rho-\tilde{\rho})
\]
has compact support in $\mathbb{R}^d\times(0,T)$ and belongs to $L^2(\mathbb{R}^d\times(0,T))$. Hence, by the standard approximation property of mollifiers,
\[
    z^\epsilon\to z
    \qquad\text{strongly in }L^2(\mathbb{R}^d\times(0,T)).
\]
Passing to the limit in \eqref{eq:gg104}, we obtain
\begin{equation}\label{eq:gg105}
    \intt\intd 
    \big|\eta_{R,\delta}(v,t)(\rho(v,t)-\tilde{\rho}(v,t))\big|^2\,dvdt=0.
\end{equation}
Therefore,
\[
    \eta_{R,\delta}(v,t)(\rho(v,t)-\tilde{\rho}(v,t))=0
    \qquad\text{for a.e. }(v,t)\in\mathbb{R}^d\times(0,T).
\]

By construction, $\eta_{R,\delta}=1$ on $B_R\times[\delta,T-\delta]$. Hence
\[
    \rho(v,t)=\tilde{\rho}(v,t)
    \qquad\text{for a.e. }(v,t)\in B_R\times[\delta,T-\delta].
\]
Since $R>0$ is arbitrary and $\delta>0$ can be chosen arbitrarily small, we let $R\to\infty$ and $\delta\to0$ to conclude that
\[
    \rho=\tilde{\rho}
    \qquad\text{a.e. in }\mathbb{R}^d\times(0,T).
\]
This proves uniqueness of weak solutions.

The condition $\epsilon\ll\delta$ is used to ensure that the time mollification does not interact with the endpoints $t=0$ and $t=T$. More precisely, for fixed $\delta>0$, we choose $\epsilon>0$ sufficiently small so that the support of the mollified function remains strictly inside $(0,T)$ on the region where the time cut-off $\chi^\delta_{[\delta,T-\delta]}$ is active. Thus $h=\eta_{R,\delta}(z^\epsilon)^\epsilon$ is a smooth compactly supported function in $\mathbb{R}^d\times(0,T)$, and it is admissible in the duality argument above.}
\begin{remark}\label{rmk:rmk220}
The Galerkin approximation procedure described above also applies when the space $\mathbb{R}^d$ is substituted with a compact manifold $\mathcal{M}$ {without boundary}. Therefore, the main theorem remains valid for equations \eqref{eq:g1} and \eqref{eqg:2} on smooth compact manifolds. In fact, when $\mathcal{M}$ is a smooth compact manifold {without boundary}, we can simply set $B_i=\mathcal{M}$ and utilize the first $K$ eigenfunctions of the Laplace-Beltrami operator on $\mathcal{M}$ in place of $\rrevised{\mathcal{H}_{i,K}}$. Consequently, all the analyses from the previous sections continue to hold.
\end{remark}

\rrevised{\section{Applying the main theorem to CBO}}\label{sec:cbo}
Here we make the following assumption: the cost function $f: \mathbb{R}^d \rightarrow \mathbb{R}$ is bounded from below with $\underline{f}:=\inf f$. Moreover, there exist constants $L_f,M,c_l$ and $c_u>0$ such that
\begin{equation}\label{condition:f}
    \begin{cases}|f(v)-f(u)| \leq L_f(\norm{v}+\norm{u})\norm{v-u} & \text { for all } v, u \in \mathbb{R}^d, \\ f(v)-\underline{f} \leq c_u\left(1+\norm{v}^2\right) & \text { for all } v\in \mathbb{R}^d ,\\
f(v)-\underline{f}\geq c_l\norm{v}^2& \text { for all } \norm{v}\geq M. \end{cases}
\end{equation}
Then \cite[Theorem 3.2]{carrillo2018analytical} ensures that the CBO dynamics
\begin{equation}\label{eq:gcbo}
		\begin{aligned}
			&\partial_t\rho=\div\Big((v-v_\alpha(\rho,t))\rho\Big)+\Delta\Big(\norm{v-v_\alpha(\rho,t)}^2\rho\Big)\\
			&\rho(v,0)=\varrho(v),
		\end{aligned}
	\end{equation}
    here $v_\alpha(\rho,t):=\intd ve^{-\alpha f(v)}d\rho(v,t)/\intd e^{-\alpha f(v)}d\rho(v,t)$,
    has a solution in $\mathcal{C}(0,T,\mathcal{P}(\mathbb{R}^d))$~(see Remark \ref{rmke:2} below), which is also a weak solution defined in this work by It\^o's lemma. We denote this solution as $\bar{\rho}$ and $v_\alpha(\bar{\rho},t)$ is $1/2$-H\"older continuous in time, see also \cite[Theorem 3.2]{carrillo2018analytical}.
 Let $v_\alpha(t):=v_\alpha(\bar{\rho},t)$. Then, by Theorem \ref{thm:main}, the following linear equation
	\begin{equation}
		\begin{aligned}
			&\partial_t\rho=\div\Big((v-v_\alpha(t))\rho\Big)+\Delta\Big(\norm{v-v_\alpha(t)}^2\rho\Big)\\
			&\quad=\div\Big(\norm{v-v_\alpha(t)}^2\nabla\rho\Big)+3\inner{v-v_\alpha(t)}{\nabla\rho}+3d\rho,\\
			&\rho(v,0)=\varrho(v),
		\end{aligned}\label{eq:CBOrev}
	\end{equation}
    where $G(v,t):=\norm{v-v_\alpha(t)}^2,\J(v,t):=v-v_\alpha(t),g(v,t)=0$, satisfies Assumptions \ref{asp:g1}, \ref{asp:g2} and \ref{asp:g3}. By Remark \ref{rmke:g211} equation \eqref{eq:CBOrev} has a unique weak solution $\rho$ and $\rho\in\Gamma_G$, thus $\bar{\rho}=\rho\in\Gamma_G$.

	
\subsection{Uniqueness of the CBO dynamics}
	Now, we show that the measure solution to the non-linear equation \eqref{eq:gcbo} is unique.
\begin{definition}
    We say $\rho$ is a measure solution to the CBO dynamics, if $\rho\in\mathcal{C}(0,T;\mathcal{P}(\mathbb{R}^d))$ and for any $\psi(v,t)\in \mathcal{C}^{2,1}(\mathbb{R}^d\times[0,T])$ with $\norm{\nabla\psi},\norm{\partial_i\partial_j\psi}$ bounded, we have
    \begin{equation}
        \begin{aligned}
            &\int_{0}^{s}\intd\Big(\partial_t\psi(v,t)+\Delta \psi(v,t)\norm{v-v_\alpha(\rho,t)}^2d\rho(v,t)-\inner{\nabla \psi(v,t)}{v-v_\alpha(\rho,t)}\Big)d\rho(v,t)dt\\
            &=\intd\psi(v,s)d\rho(v,s)-\intd\psi(v,0)d\varrho(v),
        \end{aligned}
    \end{equation}
    for any $s\in [0,T]$.
\end{definition}
\begin{remark}\label{rmke:2}
    What we mean  with $\rho\in\mathcal{C}(0,T;\mathcal{P}(\mathbb{R}^d))$ is more explicitly explained as follows:  for each $t\in [0,T]$, $\rho(\cdot,t)\in\mathcal{P}(\mathbb{R}^d)$, the space of probability measure, and $\intd h(v)d\rho(v,t)$ is continuous for any $h\in \mathcal{C}_b(\mathbb{R}^d)$, the space of bounded continuous function on $\mathbb{R}^d$.
\end{remark}

By Sobolev embedding, it is direct to derive that if $\psi\in\Gamma_G$, then $\psi$ satisfies the condition in the above definition. Therefore, if $\rho$ is a measure solution, then it is also a weak solution defined by Definition \ref{def:sima24}.

	Let $\rho$ be a measure solution and choose the test functions $e^{-\alpha f(v)}$ and $v_ie^{-\alpha f(v)},i=1,\ldots,d$ which are all bounded on $\mathbb{R}^d$, 
    so that $\intd e^{-\alpha f(v)}d\rho(v,t)$ and $\intd ve^{-\alpha f(v)}d\rho(v,t)$ are continuous on $[0,T]$, since $\rho\in\mathcal{C}(0,T;\mathcal{P}(\mathbb{R}^d))$; combine this with the fact that for each $t\in [0,T]$, $\|v_\alpha(\rho,t)\|<\infty$, then we have $v_{\alpha}(\rho,t):=\intd ve^{-\alpha f(v)}d\rho(v,t)/\intd e^{-\alpha f(v)}d\rho(v,t)$ is continuous on $[0,T]$. Thus by Theorem \ref{thm:main}, for this fixed $v_{\alpha}(\rho,t)$, the following equation 
    \begin{equation}\label{eq:g127}
		\begin{aligned}
			&\partial_t\dot{\rho}=\Delta\Big(\norm{v-v_\alpha(\rho,t)}\dot{\rho}\Big)+\div\Big((v-v_\alpha(\rho,t))\dot{\rho}\Big),\\
			&\dot{\rho}(v,0)=\varrho,
		\end{aligned}
	\end{equation}
    has a unique weak solution, which we denote as $\dot{\rho}$. Such solution belongs to $\Gamma_G$, because $\rho$ is a measure solution to equation \eqref{eq:g127}, so it is also a weak solution, hence $\rho=\dot{\rho}\in\Gamma_G$. Since the law of the process $V_t$, which satisfies the following SDE
	\begin{equation}\label{eq:g126}
		dV_t=-(V_t-v_{\alpha}(\rho,t))dt+\norm{V_t-v_\alpha(\rho,t)}dB_t,
	\end{equation}
    is also a measure solution~(by It\^o's lemma, thus a weak solution) of \eqref{eq:g127}, so $\rho$ must be the law of the process $V_t$. So in summary, we have shown that any measure solution to the non-linear PDE \eqref{eq:gcbo} is in function class $\Gamma_G$ and is the law of the \rrevised{distribution-dependent} SDE \rrevised{\begin{equation}
		dV_t=-(V_t-v_{\alpha}(\operatorname{Law}(V_t),t))dt+\norm{V_t-v_\alpha(\operatorname{Law}(V_t),t)}dB_t,
	\end{equation}} and vice versa by \cite[Theorem 3.2]{carrillo2018analytical}. By \cite[Theorem 3.2]{carrillo2018analytical}, the nonlinear SDE has a unique strong solution, so the nonlinear PDE \eqref{eq:gcbo} also has a unique measure solution, and this measure solution is in $\Gamma_G$.

\subsection{Positivity of CBO solutions for  $d>1$}
 Since the measure solution of equation \eqref{eq:gcbo} is unique and in the class $\Gamma_G$, in what follows, for simplicity, we will write $v_\alpha(\rho,t)$ as $v_\alpha(t)$, thus equation \eqref{eq:gcbo} becomes
\begin{equation}
\partial_t\rho(v,t)=\div((v-v_{\alpha}(t))\rho(v,t))+\Delta(\norm{v-v_{\alpha}(t)}^2\rho(v,t)).
\end{equation}
Let $v=v+v_{\alpha}(t)$, then we have
\begin{equation}\label{eq:665}
\partial_t\rho(v+v_{\alpha}(t),t)=\div(v\rho(v+v_{\alpha}(t),t))+\Delta(\norm{v}^2\rho(v+v_{\alpha}(t),t)).
\end{equation}
Now, let $\tilde{\rho}(v,t):=\rho(v+v_{\alpha}(t),t)$, then
\begin{equation}
\begin{aligned}
    \partial_t\tilde{\rho}(v,t)&=\frac{d}{dt}\rho(v+v_{\alpha}(t),t)\\
    &=\partial_t\rho(v+v_{\alpha}(t),t)+\inner{\nabla\rho(v+v_{\alpha}(t),t)}{\frac{d}{dt}v_\alpha(t)}\\
    &=\div(v\rho(v+v_{\alpha}(t),t))+\Delta(\norm{v}^2\rho(v+v_{\alpha}(t),t))\\
    &\quad+\inner{\nabla\rho(v+v_{\alpha}(t),t)}{\frac{d}{dt}v_\alpha(t)}\\
    &=\div(v\tilde{\rho}(v,t))+\Delta(\norm{v}^2\tilde{\rho}(v,t))+\inner{\nabla\tilde{\rho}(v,t)}{P(t)}\\
    &=\norm{v}^2\Delta\tilde{\rho}(v,t)+\inner{5v+P(t)}{\nabla\tilde{\rho}(v,t)}+3d\tilde{\rho}(v,t),
\end{aligned}
\end{equation}
here $P(t):=\frac{d}{dt}v_\alpha(t)$~(we can prove this is bounded by assuming the growth rate of $\nabla f,\Delta f$  at most polynomial, see Lemma \ref{lem:g211} below). 
At this point, we have to assume $d>1$.
Then on $B_R(0)\setminus B_{\epsilon}(0)\times [0,T]$ for any $0<\epsilon<R<\infty$, we have $\tilde{\rho}$ is the classical solution to 
\begin{equation}
\partial_t\tilde{\rho}(v,t)=\norm{v}^2\Delta\tilde{\rho}(v,t)+\inner{5v+P(t)}{\nabla\tilde{\rho}(v,t)}+3d\tilde{\rho}(v,t),
\end{equation}
and on $B_R(0)\setminus B_{\epsilon}(0)\times [0,T]$, the equation is uniform parabolic and then by the parabolic Harnack inequality~ see \cite[Section 7.1, Theorem 10]{evans2022partial}, we have
\begin{equation}
\sup_{B_R(0)\setminus B_{\epsilon}(0)} \tilde{\rho}(\cdot,t_1)\leq C\inf_{B_R(0)\setminus B_{\epsilon}(0)} \tilde{\rho}(\cdot,t_2),
\end{equation}
for any $0<t_1<t_2\leq T$. If $\rho(v,0)=\varrho$ is not identically $0$, then we can find some $v\in \mathbb{R}^d$, such that $\varrho(B_r(v))>0$. Then by \cite[Proposition 23]{B1-fornasier2021global}, we have $\rho(B_r(v),t_1)>0$ for any $t_1>0$ and thus we can find $R>0$ large enough and $\epsilon<0$ small enough, such that
\begin{equation}
\sup_{B_R(0)\setminus B_{\epsilon}(0)} \tilde{\rho}(\cdot,t_1)>0,
\end{equation}
and thus by the parabolic Harnack inequality, we have
\begin{equation}
\inf_{B_R(0)\setminus B_{\epsilon}(0)} \tilde{\rho}(\cdot,t_2)>0,
\end{equation}
for any $t_2>t_1>0$. Since $R,\epsilon,t_1,t_2$ are arbitrary, so we can conclude that $\rho(v,t)>0$ for any $v\not= v_\alpha(t)$, for any $t>0$.
\begin{remark}
    To apply the parabolic Harnack inequality, we need $B_R(0)\setminus B_\epsilon(0)$ to be connected, this is only possible when $d>1$. Actually, when $d=1$, for the equation
    \begin{eqnarray}
        \partial_t\rho(v,t)=\div((v-v^*)\rho(v,t))+\Delta(\norm{v-v^*}^2\rho(v,t))
    \end{eqnarray}
    where $v^*$ is any fixed value in $\mathbb{R}$, if the initial distribution $\rho(0,\cdot)=\varrho$ has support included in $(-\infty,v^*)$, then $\rho(t,v)=0$, for any $t\geq 0,v>v^*$. To prove it, just choose the test function $\phi(v)=\chi_{[v^*,\infty)}(v)$, here $\chi_{[v^*,\infty)}$ is the indicator function of the set $[v^*,\infty)$. Then by using integration by parts~(which is applicable, since on the boundary $v-v^*=0$), we have
    \begin{eqnarray}
        \frac{d}{dt}\int \phi(v)d\rho(t,v)=0,
    \end{eqnarray}
    thus $\int \phi(v)d\rho(t,v)=0$, for any $t\geq 0$.
\end{remark}
\begin{lemma}\label{lem:g211}
Assume 
\begin{equation}\label{condition:ff}
    \begin{aligned}
        &\norm{\Delta f}\leq C(1+\norm{v}^p),\\
        &\norm{\nabla f}\leq C(1+\norm{v}^q),\text{ for some polynomial order of $p,q>0$}.
    \end{aligned}
\end{equation}
Then it holds $\norm{\frac{d}{dt}v_{\alpha}(t)}<C<\infty$.
\end{lemma}

\begin{proof}
Use It\^o formula to functions $e^{-\alpha f(v)},ve^{-\alpha f(v)}$, and obtain 
\begin{equation}
    \frac{d}{dt}\frac{\intd ve^{-\alpha f(v)}d\rho_t}{\intd e^{-\alpha f(v)}d\rho_t}=\frac{\frac{d}{dt}\intd ve^{-\alpha f(v)}d\rho_t}{\intd e^{-\alpha f(v)}d\rho_t}-v_{\alpha}(\rho_t)\frac{\frac{d}{dt}\intd e^{-\alpha f(v)}d\rho_t}{\intd e^{-\alpha f(v)}d\rho_t},
\end{equation}
where
\begin{equation}
    \begin{aligned}
        \frac{d}{dt}\intd ve^{-\alpha f(v)}d\rho_t&=-\intd \inner{\nabla(ve^{-\alpha f(v)})}{v-v_{\alpha}(\rho_t)}d\rho_t+\intd\Delta(ve^{-\alpha f(v)})\norm{v-v_{\alpha}(\rho_t)}^2d\rho_t\\
        &=-\intd (v-v_{\alpha}(\rho_t))e^{-\alpha f(v)}d\rho_t+\alpha\intd\inner{\nabla f(v)}{v-v_{\alpha}(\rho_t)}ve^{-\alpha f(v)}d\rho_t\\
        &\quad+\alpha^2\intd v\norm{v-v_{\alpha}(\rho_t)}^2\norm{\nabla f(v)}^2e^{-\alpha f(v)}d\rho_t\\
        &\quad-\alpha\intd v\norm{v-v_{\alpha}(\rho_t)}^2\Delta f(v)e^{-\alpha f(v)}d\rho_t-2\alpha\intd \norm{v-v_{\alpha}(\rho_t)}^2\nabla f(v)e^{-\alpha f(v)}d\rho_t,
    \end{aligned}
\end{equation}
and
\begin{equation}
    \begin{aligned}
        \frac{d}{dt}\intd e^{-\alpha f(v)}d\rho_t&=-\intd \inner{\nabla(e^{-\alpha f(v)})}{v-v_{\alpha}(\rho_t)}d\rho_t+\intd\Delta(e^{-\alpha f(v)})\norm{v-v_{\alpha}(\rho_t)}^2d\rho_t\\
        &=\alpha\intd\inner{\nabla f(v)}{v-v_{\alpha}(\rho_t)}e^{-\alpha f(v)}d\rho_t+\alpha^2\intd\norm{v-v_{\alpha}(\rho_t)}^2\norm{\nabla f(v)}^2e^{-\alpha f(v)}d\rho_t\\
        &\quad -\alpha\intd\norm{v-v_{\alpha}(\rho_t)}^2\Delta f(v)e^{-\alpha f(v)}d\rho_t.
    \end{aligned}
\end{equation}Therefore by a direct generalization of \cite[Lemma 3.3]{carrillo2018analytical}, see \cite[Theorem 2.3, Proposition A.3]{gerber2023propagation}, and an application of the growth assumption on $\nabla f,\Delta f$, we have $\norm{\partial_tv_{\alpha}(t)}<C<\infty$.
\end{proof}
\subsection{Conclusion}
In summary in Section \ref{sec:cbo}, we proved the following theorem.
\begin{theorem}\label{thm:positivity}
    Let the objective $f$ satisfies condition \eqref{condition:f}, then the non-linear PDE \eqref{eq:gcbo} has a unique measure solution $\rho$, and this solution belongs to function class $\Gamma_G$. If we further assume condition \eqref{condition:ff}, then $\rho(v,t)>0$, for any $t>0,v\in\mathbb{R}^d\setminus \{v_\alpha(\rho,t)\}$, for $d>1$.
\end{theorem}

\begin{remark}
{For a connected smooth compact manifold $\mathcal{M}$ without boundary and embedded in $\mathbb R^d$, we can consider the following CBO equation
\begin{equation}\label{eq:cbom}
 \partial_t \rho_t=\lambda \div_{\mathcal{M}}\Big(P(v)\left(v-v_{\alpha}\left(\rho_t\right)\right) \rho_t\Big)+\frac{\sigma^2}{2} \Delta_{\mathcal{M}}\left(\|v-v_{\alpha}\left(\rho_t\right)\|^2 \rho_t\right), \quad t>0, v \in \mathcal{M},
 \end{equation}
 where $\lambda,\sigma>0$ are constants and $P(v)$ is the orthogonal projection onto the tangent space at $v \in \mathcal{M}$. The model \eqref{eq:cbom} was introduced, for instance, in \cite{fornasier2020consensus_hypersurface_wellposedness} for the case of smooth hypersurfaces (e.g., torus and hypersphere) . In this case of $\mathcal{M}$ being an hypersurface, we have  
 $P(v)=\mathrm{I}_d-(\nabla \operatorname{dist}(v,\mathcal{M}))(\nabla \operatorname{dist}(v,\mathcal{M}))^\top$, where the signed distance function is defined by $\operatorname{dist}(v,\mathcal{M}):=\inf_{u\in \mathcal{M}}\norm{v-u}$ if $v$ is on the exterior of $\mathcal{M}$, else $\operatorname{dist}(v,\mathcal{M}):=-\inf_{u\in \mathcal{M}}\norm{v-u}$, and it can be verified that all the assumptions are satisfied. Thus, by Remark \ref{rmk:rmk220}, Theorem \ref{thm:positivity} also holds to the CBO equation \eqref{eq:cbom} on $\mathcal{M}$.
}
\end{remark}

\section*{Acknowledgment} 
Funded by the European Union. Views and opinions expressed are however those of the author(s) only and do not necessarily reflect those of the European Union or the European Research Council Executive Agency. Neither the European Union nor the granting authority can be held responsible for them.
    This project has received funding from the European Research Council (ERC) under the European Union’s Horizon Europe research and innovation programme (grant agreement No. 101198055, project acronym NEITALG).

\bibliography{bib.bib}
\bibliographystyle{abbrv}

\section{\rrevised{Proofs of Lemma~\ref{lem:g25} and Lemma~\ref{lem:g4}}}	\label{apx:2}
	The following multivariable Faa Di Bruno formula is from \cite[Theorem 2]{constantine1996multivariate}. Recall a linear order on $\mathbb{N}_0^d$ is: if $\boldsymbol{\mu}=\left(\mu_1, \ldots, \mu_d\right)$ and $\boldsymbol{\nu}=$ $\left(\nu_1, \ldots, \nu_d\right)$ are in $\mathbb{N}_0^d$, we write $\boldsymbol{\mu} \prec	\boldsymbol{\nu}$ provided one of the following holds:
(i) $|\boldsymbol{\mu}|<|\boldsymbol{\nu}|$;
(ii) $|\boldsymbol{\mu}|=|\boldsymbol{\nu}|$ and $\mu_1<\nu_1$; or
(iii) $|\boldsymbol{\mu}|=|\boldsymbol{\nu}|, \mu_1=\nu_1, \ldots, \mu_k=\nu_k$ and $\mu_{k+1}<\nu_{k+1}$ for some $1 \leq k<d$.

Let $h(v)=z(\boldsymbol{w}(v))=z \circ \boldsymbol{w}(v)$, where $\boldsymbol{w}(v)=\left(w_1(v),\ldots,w_m(v)\right)\in \mathbb{R}^m,v\in\mathbb{R}^d$, the multivariable Faa Di Bruno formula reads:
\begin{equation}
	h^{(\boldsymbol{\nu})}=\sum_{1 \leq|\boldsymbol{\lambda}| \leq n} z^{(\boldsymbol{\lambda})}  \sum_{p(\boldsymbol{\nu}, \boldsymbol{\lambda})}(\boldsymbol{\nu}!) \prod_{j=1}^n \frac{\left[\boldsymbol{w}^{(\boldsymbol{\ell}_j)}\right]^{\mathbf{k}_j}}{\left(\mathbf{k}_{j}!\right)\left[\boldsymbol{\ell}_{j}!\right]^{\left|\mathbf{k}_j\right|}},
\end{equation}
where $h^{(\boldsymbol{\nu})}, z^{(\boldsymbol{\lambda})}, \boldsymbol{w}^{\left(\boldsymbol{\ell}_j\right)}$ are evaluated in $v, \boldsymbol{w}(v), v$ respectively. Here $\boldsymbol{\nu}\in\mathbb{N}_0^d,\boldsymbol{\lambda}\in\mathbb{N}_0^m, n=|\boldsymbol{\nu}|$ and

\begin{equation}
	\begin{aligned}
		& p(\boldsymbol{\nu}, \boldsymbol{\lambda})=\Big\{\left(\mathbf{k}_1 \ldots, \mathbf{k}_n ; \boldsymbol{\ell}_1, \ldots, \boldsymbol{\ell}_n\right), \mathbf{k}_i\in\mathbb{N}_0^m,\boldsymbol{\ell}_i\in\mathbb{N}_0^d: \text { for some } 1 \leq s \leq n, \\
		& \quad \mathbf{k}_i=\mathbf{0} \text { and } \boldsymbol{\ell}_i=\mathbf{0} \text { for } 1 \leq i \leq n-s ;\left|\mathbf{k}_i\right|>0 \text { for } n-s+1 \leq i \leq n ; \\
		& \quad \text { and } \mathbf{0} \prec \boldsymbol{\ell}_{n-s+1} \prec \cdots \prec \boldsymbol{\ell}_n \text { are such that }  \sum_{i=1}^n \mathbf{k}_i=\boldsymbol{\lambda}, \sum_{i=1}^n\left|\mathbf{k}_i\right| \boldsymbol{\ell}_i=\boldsymbol{\nu}\Big\} .
	\end{aligned}
\end{equation}

With the above formula, we can prove the following lemma.
\begin{lemma}\label{lem:g1}Under Assumptions \ref{asp:g1} and \ref{asp:g2}, we have $\norm{\Big(G(\frac{R_i v}{\norm{v}})\Big)^{(\boldsymbol{\nu})}}\leq C\Big(1+G(\frac{R_i v}{\norm{v}})\Big)$, for any $v\in B_{R_i-1}^c(0)$ and any vector $\boldsymbol{\nu}=\left(\nu_1, \ldots, \nu_d\right) \in \mathbb{N}_0^d$. When $|\boldsymbol{\nu}|=1$, we have $\norm{\Big(G(\frac{R_i v}{\norm{v}})\Big)^{(\boldsymbol{\nu})}}\leq CG^{1/2}(\frac{R_i v}{\norm{v}})\Big(1+G^{1/2}(\frac{R_i v}{\norm{v}})\Big)$, for any $v\in B_{R_i-1}^c(0)$. Here constants may depend on $k$ but do not depend on $i$.
\end{lemma}
\begin{proof}
	When $|\boldsymbol{\nu}|=0$, the result is immediate; 	we will focus on the rest of the cases. Let us consider $v\in B^c_{R_i-1}(0)$.
	
	First, for each $\boldsymbol{\ell}_j\in\mathbb{N}_0^d$,  it is easy to verify that for each $k\in [d]$, we have 
	\begin{equation}\label{eq:g74}
		\norm{\Big(\frac{R_i v_k}{\norm{v}}\Big)^{(\boldsymbol{\ell}_j)}}\leq C\frac{R_i}{\norm{v}^{|\boldsymbol{\ell}_j|}},
	\end{equation}
	and thus let $\boldsymbol{w}(v):=\frac{R_i v}{\norm{v}}$ with $m=d$, we have 
	\begin{equation}
		\norm{ \prod_{j=1}^n \frac{\left[\boldsymbol{w}^{(\boldsymbol{\ell}_j)}\right]^{\mathbf{k}_j}}{\left(\mathbf{k}_{j}!\right)\left[\boldsymbol{\ell}_{j}!\right]^{\left|\mathbf{k}_j\right|}}}\leq  \prod_{j=1}^n\frac{CR^{|\boldsymbol{k}_j|}_i}{\norm{v}^{|\boldsymbol{\ell}_j||\boldsymbol{k}_j|}}\leq \frac{CR_i^{|\boldsymbol{\lambda}|}}{\norm{v}^{|\boldsymbol{\nu}|}}\leq C,
	\end{equation}
	since $|\boldsymbol{\lambda}| \leq n=|\boldsymbol{\nu}|$ and $v\in B_{R_i-1}^c(0)$.
	
	Second, by the Assumption \ref{asp:g1}, we have 
	\begin{equation}
		\norm{G^{(\boldsymbol{\lambda})}(\frac{R_i v}{\norm{v}})}\leq C\Big(1+G(\frac{R_i v}{\norm{v}})\Big).
	\end{equation}
	
	So using the Faa Di Bruno formula and the above two inequalities, we have
	\begin{equation}
		\norm{\Big(G(\frac{R_i v}{\norm{v}})\Big)^{(\boldsymbol{\nu})}}\leq C\Big(1+G(\frac{R_i v}{\norm{v}})\Big).
	\end{equation}
	
	When $|\boldsymbol{\nu}|=1$, without loss of generality, we will assume $\Big(G(\frac{R_i v}{\norm{v}})\Big)^{(\boldsymbol{\nu})}=\frac{\partial}{\partial v_1}G(\frac{R_i v}{\norm{v}})$, then we have
	\begin{equation}
		\frac{\partial}{\partial v_1}G(\frac{R_i v}{\norm{v}})=\sum_{k=1}^d\frac{\partial}{\partial v'_k}G(v')\mid_{v'=\frac{R_iv}{\norm{v}}}\frac{\partial}{\partial v_1}\frac{R_iv_k}{\norm{v}},
	\end{equation}
	then using \eqref{eq:g74} and Assumption \ref{asp:g1}, we get $\norm{\Big(G(\frac{R_i v}{\norm{v}})\Big)^{(\boldsymbol{\nu})}}\leq CG^{1/2}(\frac{R_i v}{\norm{v}})\Big(1+G^{1/2}(\frac{R_i v}{\norm{v}})\Big)$.

\end{proof}

\begin{lemma}\label{lem:gg2}Under Assumptions \ref{asp:g1} and \ref{asp:g2}, we have $\norm{\Big(\sqrt{G(\frac{R_i v}{\norm{v}})+1}\Big)^{(\boldsymbol{\nu})}}\leq C\sqrt{(1+G(\frac{R_i v}{\norm{v}}))}$, for any $v\in B_{R_i-1}^c(0)$ and any vector $\boldsymbol{\nu}=\left(\nu_1, \ldots, \nu_d\right) \in \mathbb{N}_0^d$. Here constants may depend on $k$ but do not depend on $i$.
\end{lemma}
\begin{proof}
	When $\boldsymbol{\nu}=0$, the result is immediate; we will focus on the rest of the cases. Let us consider $v\in B^c_{R_i-1}(0)$.
	
	In this case, the Faa Di Bruno formula reads:
	\begin{equation}
		\Big(\sqrt{G(\frac{R_i v}{\norm{v}})+1}\Big)^{(\boldsymbol{\nu})}=\sum_{1 \leq\lambda \leq n} z^{({\lambda})} \ \sum_{p(\boldsymbol{\nu}, {\lambda})}(\boldsymbol{\nu}!) \prod_{j=1}^n \frac{\left[{w}^{(\boldsymbol{\ell}_j)}\right]^{{k}_j}}{\left({k}_{j}!\right)\left[\boldsymbol{\ell}_{j}!\right]^{{k}_j}},
	\end{equation}
	here $z(\cdot)=\sqrt{\cdot}, w(v)=G(\frac{R_iv}{\norm{v}})+1$ and $m=1$.
	
	First, use Lemma \ref{lem:g1}, and we have 
	\begin{equation}
		\norm{\prod_{j=1}^n \frac{\left[{w}^{(\boldsymbol{\ell}_j)}\right]^{{k}_j}}{\left({k}_{j}!\right)\left[\boldsymbol{\ell}_{j}!\right]^{{k}_j}}}\leq C\Big(1+G(\frac{R_i v}{\norm{v}})\Big)^{\sum_{j=1}^nk_j}=C\Big(1+G(\frac{R_i v}{\norm{v}})\Big)^{\lambda}.
	\end{equation}
	
	Second, we have
	\begin{equation}
		\norm{z^{(\lambda)}(x)\mid_{x=G(\frac{R_i v}{\norm{v}})+1}}= C\Big(G(\frac{R_i v}{\norm{v}})+1\Big)^{\frac{1}{2}-\lambda}.
	\end{equation}

	So, finally, we obtain $\norm{\Big(\sqrt{G(\frac{R_i v}{\norm{v}})+1}\Big)^{(\boldsymbol{\nu})}}\leq C\sqrt{(1+G(\frac{R_i v}{\norm{v}}))}$.
\end{proof}

Next, let 
\begin{equation}
	\varphi(x)= \begin{cases}e^{-1 /\left(1-|x|^2\right)} / I & \text { if }|x|<1, \\ 0 & \text { if }|x| \geq 1,\end{cases}
\end{equation}
here $x\in\mathbb{R}$ and $I$ is the normalization constant, and $\varphi_\epsilon(x):=C(\epsilon)\varphi(\frac{x}{\epsilon})$, here $C(\epsilon)$ is the normalization constant.   We define

\begin{equation}
	S(x):=\Big(\chi_{[\frac{1}{2},\infty)}*\varphi_{\frac{1}{8}}\Big)(x),
\end{equation}
where $\chi$ is the characteristic function. Then it is easy to verify that
\begin{eqnarray}
	&S(x)\in [0,1], \forall x\in \mathbb{R};\\
	&S(x)=0, S^{(k)}(x)=0, \forall x\leq 0, k\geq 1;\\
	&S(x)=1,S^{(k)}(x)=0,\forall x\geq 1, k\geq 1;\\
	&|S^{(k)}(x)|\leq C, \forall x\in \mathbb{R}.
\end{eqnarray}

Define $S_i(v):=S(\norm{v}-R_i+1)$, then it is easy to verify that 
\begin{eqnarray}
	&S_i(v)\in [0,1], \forall v\in \mathbb{R}^d;\\
	&S_i(v)=0, S_i^{(k)}(v)=0, \forall v\in B_{R_i-1}(0), k\geq 1;\\
	&S_i(v)=1,S_i^{(k)}(v)=0,\forall v\in B^c_{R_i}(0), k\geq 1;\\
	&\norm{S_i^{(k)}(v)}\leq C, \forall v\in B_{R_i}(0)\setminus B_{R_i-1}(0),k\geq 1,
\end{eqnarray}
here constants may depend on $k$ but do not depend on $i$.
Define
\begin{eqnarray}
	&\overline{G}_i(v):=G(v)(1-S_i(v))+\Big(1+G(\frac{R_iv}{\norm{v}})\Big)S_i(v),\\
	&\overline{\J}_i(v):=\J(v)(1-S_i(v))+\sqrt{G(\frac{R_iv}{\norm{v}})+1}\e S_i(v),
\end{eqnarray}
here $e:=(1,\ldots,1)\in \mathbb{R}^d$.
Then we have the following lemma.
\begin{lemma}\label{lem:gg25}Under Assumptions \ref{asp:g1} and \ref{asp:g2}, we have for any $v\in\mathbb{R}^d$ that
	\begin{eqnarray}
		&\norm{\overline{G}_i^{(1)}}\leq C\overline{G}^{\frac{1}{2}}_i\Big(1+\overline{G}^{\frac{1}{2}}_i\Big);\\
		&\norm{\overline{G}_i^{(k)}}\leq C\Big(1+{\overline{G}_i}\Big),\quad \forall k\geq 2;\\
		&\norm{\overline{\J}_i}^2\leq C\overline{G}_i;\\
		&\norm{\overline{\J}_i^{(k)}}\leq C(1+\overline{G}^{\frac{1}{2}}_i)\quad \forall k\geq 1,
	\end{eqnarray}
	where the constants are uniform for any $t\in [0,T]$ and may depend on $k$ but do not depend on $i$.
\end{lemma}
\begin{proof}
	{When $v\in B_{R_i-1}(0)$, we have $\overline{G}_i(v)=G(v),\overline{\J}_i(v)=\J(v)$, so the result is a direct consequence of Assumption \ref{asp:g1}; when $v\in B^c_{R_i}(0)$, we have 
    $\overline{G}_i(v)=1+G(\frac{R_iv}{\norm{v}}),\overline{\J}_i(v)=\sqrt{G(\frac{R_iv}{\norm{v}})+1}\e$, the result is a direct consequence of Assumption \ref{asp:g1}, \ref{asp:g2}, \ref{lem:g1}, and Lemma \ref{lem:gg2}.}
    
    So in the next, we will only focus on the case when $v\in B_{R_i}(0)\setminus B_{R_i-1}(0)$.
	In the following , for simplicity, we will denote $G_i'(v):=1+G(\frac{R_iv}{\norm{v}})$.
	
	First, we have
	\begin{equation}
		\begin{aligned}
			\norm{\overline{G}_i^{(1)}}&=\norm{G^{(1)}(1-S_i)-GS_i^{(1)}+G_i'^{(1)}S_i+G_i'S_i^{(1)}}\\
			&\leq C\Big[\norm{G^{(1)}}+G+\norm{G_i'^{(1)}}+G_i'\Big]\\
			&\leq C\Big[G^{\frac{1}{2}}(1+G^{\frac{1}{2}})+G_i'^{\frac{1}{2}}(1+G_i'^{\frac{1}{2}})\Big]\\
			&\leq C\overline{G}^{\frac{1}{2}}_i\Big(1+\overline{G}^{\frac{1}{2}}_i\Big),
		\end{aligned}
	\end{equation}
	the last inequality is due to Assumption \ref{asp:g2}: in fact, because of Assumption \ref{asp:g2}, we have 
	\begin{eqnarray}
		G&\leq C\overline{G}_i,\\
		G_i'&\leq C\overline{G}_i.
	\end{eqnarray}
	
	Second, by the Leibniz rule, we have
	\begin{equation}
		\begin{aligned}
			\norm{\overline{G}_i^{(k)}}&\leq C\Big[\sum_{s=0}^k\norm{{G}^{(s)}(1-S_i)^{(k-s)}}+\sum_{s=0}^k\norm{G_i'^{(s)}S_i^{(k-s)}}\Big]\\
			&\leq C\Big[\sum_{s=0}^k\norm{{G}^{(s)}}+\sum_{s=0}^k\norm{G_i'^{(s)}}\Big]\\
			&\leq C\Big[1+G+1+G_i'\Big]\\
			&\leq C(1+\overline{G}_i).
		\end{aligned}
	\end{equation}
	
	Third, by the convexity of $\norm{\cdot}^2$, we have
	\begin{equation}
		\begin{aligned}
			\norm{\overline{\J}_i}^2&\leq \norm{\J}^2(1-S_i)+dG_i'S_i\\
			&\leq CG(1-S_i)+dG_i'S_i\\
			&\leq C\Big(G(1-S_i)+G_i'S_i\Big)\\
			&\leq C\overline{G}_i,
		\end{aligned}
	\end{equation}
	where we used Assumption \ref{asp:g1}.
	
	Lastly,  by the Leibniz rule, Assumption \ref{asp:g1}, and Lemma \ref{lem:gg2}, we have 
	\begin{equation}
		\begin{aligned}
			\norm{\overline{\J}_i^{(k)}}&\leq C\Big[\sum_{s=0}^k\norm{{\J}^{(s)}(1-S_i)^{(k-s)}}+\sum_{s=0}^k\norm{(\sqrt{G_i'})^{(s)}eS_i^{(k-s)}}\Big]\\
			&\leq C\Big[\sum_{s=0}^k\norm{{\J}^{(s)}}+\sum_{s=0}^k\norm{(\sqrt{G_i'})^{(s)}}\Big]\\
			&\leq C\Big[1+G^{\frac{1}{2}}+\sqrt{G_i'}\Big]\\
			&\leq C\Big(1+\overline{G}_i^{\frac{1}{2}}\Big).
		\end{aligned}
	\end{equation}
	
\end{proof}	
Next, we define 
\begin{equation}
	\begin{aligned}
		&\H(x)=(\chi_{[-10,10]}*\varphi)(x),\\
		&\H_i(x)=\H(\frac{x}{n_i}),
	\end{aligned}
\end{equation}
then it is direct to verify that 
\begin{itemize}
	\item ${\H_i(x)}\in [0,1],\quad\forall x\in \mathbb{R}$;
	\item  For any $k\geq 1$, $\H_i^{(k)}$ has compact support and $\norm{\H^{(k)}_i}<\frac{C}{n_i^k}$, here $C$ does not depend on $i$.
\end{itemize}

With $H_i$, we define
\begin{eqnarray}
	G_i(v)&:=H^2_i(\norm{v})\overline{G}_i(v),\\
	\J_i(v)&:=H_i(\norm{v})\overline{\J}_i(v),
\end{eqnarray}
then $G_i$ and $J_i$, which have support within $B_{12n_i}(0)$, satisfy the following lemma.
\begin{lemma}\label{lem:gg4}Under Assumptions \ref{asp:g1} and \ref{asp:g2}, we have for any $v\in\mathbb{R}^d,t\in [0,T]$ that
	\begin{eqnarray}
		&\norm{{G}_i^{(1)}}\leq C{G}^{\frac{1}{2}}_i\Big(1+{G}^{\frac{1}{2}}_i\Big);\\
		&\norm{{G}_i^{(k)}}\leq C\Big(1+{{G}_i}\Big),\quad \forall k\geq 2;\\
		&\norm{{\J}_i}^2\leq C{G}_i;\\
		&\norm{{\J}_i^{(k)}}\leq C(1+{G}^{\frac{1}{2}}_i)\quad \forall k\geq 1,
	\end{eqnarray}
	here constants are uniform for any $t\in [0,T]$ and may depend on $k$ but do not depend on $i$.
\end{lemma}
\begin{proof}
	Remember $H_i\in [0,1]$  and $\norm{H_i^{(k)}}\leq \frac{C}{n_i^k}$.
	
	First, by Lemma \ref{lem:gg25}, we have 
	\begin{equation}
		\begin{aligned}
			\norm{G_i^{(1)}}&\leq 2H_i\overline{G}_i \norm{H_i^{(1)}}+H_i^2\norm{\overline{G}^{(1)}_i}\\
			&\leq CH_i\overline{G}_i^{\frac{1}{2}} \frac{\overline{G}_i^{\frac{1}{2}}}{n_i}+CH_i^2\overline{G}^{\frac{1}{2}}_i\Big(1+\overline{G}^{\frac{1}{2}}_i\Big)\\
			&\leq CG_i^{\frac{1}{2}}\Big(1+G_i^{\frac{1}{2}}\Big),
		\end{aligned}
	\end{equation}
	since $G\leq n_i,\forall v\in B_{R_i}(0)$, so $\overline{G}_i\leq n_i+1,\forall v\in\mathbb{R}^d$, and 
	${\overline{G}_i^{{1}/{2}}}/{n_i}\leq C,\quad\forall v\in \mathbb{R}^d$.
	
	Second, by the Leibniz rule, we have by Lemma \ref{lem:gg25} that
	\begin{equation}
		\begin{aligned}
			\norm{G_i^{(k)}}&\leq C\Big[H_i^2\norm{\overline{G}_i^{(k)}}+\sum_{s=1}^k\norm{(H_i^2)^{(s)}\overline{G}_i^{(k-s)}}\Big]\\
			&\leq C\Big[H_i^2(1+\overline{G}_i)+\sum_{s=1}^k\frac{1+\overline{G}_i}{n_i}\Big]\\
			&\leq C(1+G_i).
		\end{aligned}
	\end{equation}
	
	Third, we have
	\begin{eqnarray}
		\norm{\J_i}^2=H_i^2\norm{\overline{\J}_i}^2\leq CH_i^2\overline{G}_i=CG_i.
	\end{eqnarray}
	
	Fourth, we have
	\begin{eqnarray}
		\begin{aligned}
			\norm{\J_i^{(k)}}&\leq C\Big[H_i\norm{\overline{\J}_i^{(k)}}+\sum_{s=1}^k\norm{H_i^{(s)}\overline{\J}_i^{(k-s)}}\Big]\\
			&\leq C\Big[H_i(1+\overline{G}_i^{\frac{1}{2}})+\sum_{s=1}^k\frac{1+\overline{G}_i^{\frac{1}{2}}}{n_i}\Big]\\
			&\leq C(1+G_i^{\frac{1}{2}}).
		\end{aligned}
	\end{eqnarray}
\end{proof}
\section{\rrevised{Verification of Assumption~\ref{asp:g3} for CBO}}\label{apx:1}
The proof directly follow the proofs of Lemmas \ref{lem:g1}, \ref{lem:g25} and \ref{lem:g4}. Let us fix any $t\in [0,T]$, and for simplicity, we denote $G(\cdot):=G(\cdot,t)$. Then we first have 
\begin{equation}
\begin{aligned}
    \norm{\frac{\partial}{\partial v_j}G(\frac{R_i v}{\norm{v}})}&=\sum_{k=1}^d\norm{\frac{\partial}{\partial v'_k}G(v')\mid_{v'=\frac{R_iv}{\norm{v}}}}\norm{\frac{\partial}{\partial v_j}\frac{R_iv_k}{\norm{v}}}\\
    &\leq C\Big(1+G^{\frac{1}{2}}(\frac{R_iv}{\norm{v}})\Big),
\end{aligned}
\end{equation}
by the condition in Remark \ref{rmke:g211}. 

Secondly, let $G_i':=1+G(\frac{R_iv}{\norm{v}})$, then on $B_{R_i}(0)\setminus B_{R_i-1}(0)$~(the other cases are direct by the condition in Remark \ref{rmke:g211}), we have
\begin{equation}
\begin{aligned}
    \norm{\overline{G}_i^{(1)}}&=\norm{G^{(1)}(1-S_i)-GS_i^{(1)}+G_i'^{(1)}S_i+G_i'S_i^{(1)}}\\
    &\leq C\Big[\norm{G^{(1)}}(1-S_i)+\norm{G_i'^{(1)}}S_i\Big]+C\norm{G-G_i'}\\
    &\leq C\Big[\Big(1+G^{\frac{1}{2}}\Big)(1-S_i)+\Big(1+G^{\frac{1}{2}}(\frac{R_iv}{\norm{v}})\Big)S_i\Big]\\
    &\quad+C\Big(1+\norm{\nabla G(v')}\norm{v-\frac{R_iv}{\norm{v}}}\Big),
\end{aligned}
\end{equation}
for some $v'$ between $v$ and $\frac{R_iv}{\norm{v}}$ by the mean value theorem; further, use $\norm{v-\frac{R_iv}{\norm{v}}}\leq 1$ for any $v\in B_{R_i}(0)\setminus B_{R_i-1}(0)$, we have
\begin{equation}
    \norm{\overline{G}_i^{(1)}}\leq C\Big[\Big(1+G^{\frac{1}{2}}\Big)(1-S_i)+\Big(1+G^{\frac{1}{2}}(\frac{R_iv}{\norm{v}})\Big)S_i\Big]+C\Big(1+G^{\frac{1}{2}}(v')\Big),
\end{equation}
use Assumption \ref{asp:g2}, we have for any $v\in B_{R_i}(0)\setminus B_{R_i-1}(0)$ that
\begin{eqnarray}
    &1+G\leq C\Big(1+G(\frac{R_iv}{\norm{v}})\Big),\\
    &1+G(\frac{R_iv}{\norm{v}})\leq C\Big(1+G\Big),
\end{eqnarray}
then 
\begin{eqnarray}
    &1+G\leq C\Big(1+\overline{G}_i\Big),\\
    &1+G(\frac{R_iv}{\norm{v}})\leq C\Big(1+\overline{G}_i\Big),\\
    &1+G(v')\leq C\Big(1+\overline{G}_i\Big),
\end{eqnarray}
thus
\begin{equation}
   \begin{aligned}
        &C\Big[\Big(1+G^{\frac{1}{2}}\Big)(1-S_i)+\Big(1+G^{\frac{1}{2}}(\frac{R_iv}{\norm{v}})\Big)S_i\Big]+C\Big(1+G^{\frac{1}{2}}(v')\Big)\\
        &\leq C\Big[\Big(1+(1+G)^{\frac{1}{2}}\Big)(1-S_i)+\Big(1+\Big(1+G(\frac{R_iv}{\norm{v}})\Big)^{\frac{1}{2}}\Big)S_i\Big]+C\Big(1+\Big(1+G(v')\Big)^{\frac{1}{2}}\Big)\\
        &\leq C\Big(1+\Big(1+\overline{G}_i\Big)^{\frac{1}{2}}\Big)\\
        &\leq C\Big(1+\overline{G}_i^{\frac{1}{2}}\Big),
   \end{aligned}
\end{equation}
by simple inequality $(1+x)^{1/2}\leq 1+x^{1/2},x\geq 0$, thus we proved $\norm{\overline{G}_i^{(1)}}\leq C\Big(1+\overline{G}_i^{{1}/{2}}\Big)$.

Thirdly, we have
\begin{equation}
\begin{aligned}
    \norm{G_i^{(1)}}&\leq H_i\overline{G}_i \norm{H_i^{(1)}}+H_i^2\norm{\overline{G}^{(1)}_i}\\
    &\leq H_i\overline{G}_i^{\frac{1}{2}} \frac{\overline{G}_i^{\frac{1}{2}}}{n_i}+CH_i^2\Big(1+\overline{G}^{\frac{1}{2}}_i\Big)\\
    &\leq C\Big(1+G_i^{\frac{1}{2}}\Big),
\end{aligned}
\end{equation}
since $\overline{G_i}^{1/2}/n_i\leq C$.

Finally, use the condition in Remark \ref{rmke:g211}, we have 
\begin{equation}
    \begin{aligned}
        1+\overline{G}_i(v,t_1)&:=\Big(1+G(v,t_1)\Big)(1-S_i)+\Big(2+G(\frac{R_iv}{\norm{v}},t_1)\Big)S_i\\
        &\leq C\Big[\Big(1+G(v,t_2)\Big)(1-S_i)+\Big(2+G(\frac{R_iv}{\norm{v}},t_2)\Big)S_i\Big]\\
        &=C\Big[1++G(v,t_2)(1-S_i)+\Big(1+G(\frac{R_iv}{\norm{v}},t_2)\Big)S_i\Big]\\
        &=:C\Big[1+\overline{G}_i(v,t_2)\Big],
    \end{aligned}
\end{equation}
for any $t_1,t_2\in [0,T]$, thus 
\begin{equation}
    G_i(v,t_1)\leq C\Big[H_i^2+H_i^2\overline{G}_i(v,t_2)\Big]\leq C\Big[1+G_i(v,t_2)\Big]
\end{equation}
and 
\begin{equation}
    1+G_i(v,t_1)\leq C\Big[H_i^2+H_i^2\overline{G}_i(v,t_2)\Big]\leq C\Big[1+G_i(v,t_2)\Big],
\end{equation}
for any $t_1,t_2\in [0,T]$.

\newcommand{\R}{\mathbb R}
\newcommand{\dual}[2]{\left\langle #1,#2\right\rangle}

\rrevised{\section{A Predual of \(W^{1,\infty}(0,T;H^m(B_i))\)}\label{sec:k6}

Let
\[
X:=H^m(B_i).
\]
Then \(X\) is a Hilbert space, and its dual is
\[
X'=H^{-m}(B_i).
\]
We want to identify a natural predual of
\[
W^{1,\infty}(0,T;X).
\]

Define the map
\[
\Phi:W^{1,\infty}(0,T;X)\to X\times L^\infty(0,T;X)
\]
by
\[
\Phi u=\bigl(u(0),\partial_t u\bigr).
\]
This map is linear and injective. Moreover, for every
\[
(x,g)\in X\times L^\infty(0,T;X),
\]
the formula
\[
u(t)=x+\int_0^t g(s)\,ds
\]
defines an element \(u\in W^{1,\infty}(0,T;X)\) such that
\[
u(0)=x,
\qquad
\partial_t u=g.
\]
Hence \(\Phi\) is a linear isomorphism:
\[
W^{1,\infty}(0,T;X)
\simeq
X\times L^\infty(0,T;X).
\]

Since \(X=H^m(B_i)\) is reflexive and separable, we have
\[
X=(X')',
\qquad
L^\infty(0,T;X)
=
\bigl(L^1(0,T;X')\bigr)'.
\]
Therefore
\[
X\times L^\infty(0,T;X)
=
\left(
X'\times L^1(0,T;X')
\right)'.
\]
Substituting \(X'=H^{-m}(B_i)\), we obtain the canonical predual
\[
{
\mathcal P_i
:=
H^{-m}(B_i)
\times
L^1(0,T;H^{-m}(B_i)),
}
\]
which is a normed space.

The corresponding duality pairing is
\[
\dual{(\eta,f)}{u}
=
\dual{\eta}{u(0)}_{H^{-m},H^m}
+
\int_0^T
\dual{f(t)}{\partial_t u(t)}_{H^{-m},H^m}\,dt,
\]
where
\[
(\eta,f)\in
H^{-m}(B_i)\times L^1(0,T;H^{-m}(B_i)),
\]
and
\[
u\in W^{1,\infty}(0,T;H^m(B_i)).
\]

Thus, under the identification
\[
u\longleftrightarrow \bigl(u(0),\partial_t u\bigr),
\]
we have
\[
{
W^{1,\infty}(0,T;H^m(B_i))
\simeq
\left[
H^{-m}(B_i)
\times
L^1(0,T;H^{-m}(B_i))
\right]'.
}
\]

\begin{remark}
This identification is isometric if one equips
\(W^{1,\infty}(0,T;H^m(B_i))\) with the norm
\[
\|u\|_*
=
\|u(0)\|_{H^m(B_i)}
+
\|\partial_t u\|_{L^\infty(0,T;H^m(B_i))}.
\]
This norm is equivalent to the usual norm
\[
\|u\|_{W^{1,\infty}(0,T;H^m)}
=
\|u\|_{L^\infty(0,T;H^m)}
+
\|\partial_t u\|_{L^\infty(0,T;H^m)}.
\]
Therefore the predual above is canonical up to equivalent norms.
\end{remark}

}

\rrevised{
\section{Proof of Lemma~\ref{lem:n210}}\label{apx:n210}
Since
\[
\rho_j \rightharpoonup^\ast \rho
\quad\text{in } W^{1,\infty}(0,T;H^M(\Omega)),
\]
the sequence $\{\rho_j\}_{j=1}^\infty$ is uniformly bounded in
$W^{1,\infty}(0,T;H^M(\Omega))$. Hence there exists a constant $C>0$, independent of $j$, such that
\[
\|\rho_j\|_{L^\infty(0,T;H^M(\Omega))}
+
\|\partial_t\rho_j\|_{L^\infty(0,T;H^M(\Omega))}
\le C.
\]

Since $M>m+\frac d2$, the Sobolev embedding theorem gives the compact embedding
\[
H^M(\Omega)\Subset \mathcal{C}^m(\Omega).
\]
Moreover, for any $s,t\in[0,T]$, we have
\[
\|\rho_j(t)-\rho_j(s)\|_{H^M(\Omega)}
\le
\int_s^t \|\partial_\tau\rho_j(\tau)\|_{H^M(\Omega)}\,d\tau
\le C|t-s|.
\]
By the continuous Sobolev embedding $H^M(\Omega)\hookrightarrow \mathcal{C}^m(\Omega)$, it follows that
\[
\|\rho_j(t)-\rho_j(s)\|_{\mathcal{C}^m(\Omega)}
\le C|t-s|.
\]
Thus $\{\rho_j\}$ is equicontinuous in time with values in $\mathcal{C}^m(\Omega)$.

For each fixed $t\in[0,T]$, the set
\[
\{\rho_j(t)\}_{j=1}^\infty
\]
is bounded in $H^M(\Omega)$. Since $H^M(\Omega)\Subset \mathcal{C}^m(\Omega)$, \rrevised{this same set is relatively compact in $\mathcal{C}^m(\Omega)$ by Arzel\`a--Ascoli theorem. Hence every subsequence of $\{\rho_j\}$ admits a further subsequence, still denoted by $\rho_j$, and a function
\[
\widetilde\rho\in \mathcal{C}(0,T;\mathcal{C}^m(\Omega))
\]
such that
\[
\rho_j\to \widetilde\rho
\quad\text{strongly in }\mathcal{C}(0,T;\mathcal{C}^m(\Omega)).
\]

It remains to identify $\widetilde\rho$ with $\rho$. Let
$\eta\in \mathcal{C}_c^\infty(0,T)$ and $\varphi\in \mathcal{C}^\infty(\Omega)$. Since
$\rho_j\to\widetilde\rho$ strongly in $\mathcal{C}(0,T;\mathcal{C}^m(\Omega))$, we have
\[
\int_0^T \eta(t)\int_\Omega \rho_j(x,t)\varphi(x)\,dx\,dt
\to
\int_0^T \eta(t)\int_\Omega \widetilde\rho(x,t)\varphi(x)\,dx\,dt.
\]
On the other hand, the weak-$\ast$ convergence in
$L^\infty(0,T;H^M(\Omega))$ gives
\[
\int_0^T \eta(t)\langle \rho_j(t),\varphi\rangle_{H^M,H^{-M}}\,dt
\to
\int_0^T \eta(t)\langle \rho(t),\varphi\rangle_{H^M,H^{-M}}\,dt.
\]
Therefore,
\[
\int_0^T \eta(t)\int_\Omega \widetilde\rho(x,t)\varphi(x)\,dx\,dt
=
\int_0^T \eta(t)\int_\Omega \rho(x,t)\varphi(x)\,dx\,dt.
\]
Since this holds for all such $\eta$ and $\varphi$, we obtain
\[
\widetilde\rho=\rho
\]
in the sense of distributions on $\Omega\times(0,T)$. Since both functions belong to
$\mathcal{C}(0,T;\mathcal{C}^m(\Omega))$, they are equal as functions.

Thus every subsequence of $\{\rho_j\}$ has a further subsequence converging strongly to $\rho$ in
$\mathcal{C}(0,T;\mathcal{C}^m(\Omega))$. Consequently, the whole sequence converges:
\[
\rho_j\to \rho
\quad\text{strongly in }\mathcal{C}(0,T;\mathcal{C}^m(\Omega)).
\]
This completes the proof.
}
}

\end{document}